\documentclass[11pt,letterpaper, reqno]{article}
\usepackage[utf8]{inputenc}

\usepackage[sc]{mathpazo} 
\usepackage{eulervm}
\usepackage[T1]{fontenc}

\usepackage{amsfonts, amsmath, amssymb, amsthm, enumerate, graphicx, mathtools, tikz, hyperref, enumerate,color}

\usepackage{subcaption} 

\usepackage[affil-it]{authblk}   
\usepackage[T1]{fontenc}
\usepackage[sort, numbers]{natbib}
\usepackage[letterpaper,includeheadfoot, total={6.5in,10in}]{geometry}

\usepackage{lineno} 

\usepackage{array}   
\newcolumntype{L}[1]{>{\raggedright\let\newline\\\arraybackslash\hspace{0pt}}m{#1}}
\newcolumntype{C}[1]{>{\centering\let\newline\\\arraybackslash\hspace{0pt}}m{#1}}
\newcolumntype{R}[1]{>{\raggedleft\let\newline\\\arraybackslash\hspace{0pt}}m{#1}}

\newtheorem{theorem}{Theorem}
\newtheorem{corollary}{Corollary}[theorem]
\newtheorem{lemma}[theorem]{Lemma}
\newtheorem{proposition}[theorem]{Proposition}

\newtheorem{claim}{Claim}

\usepackage{environ}
\NewEnviron{eq}{%
\begin{equation}\begin{split}
  \BODY
\end{split}\end{equation}
}

\def\sss{\scriptscriptstyle}

\newcommand{\A}{\mathrm{A}}
\newcommand{\U}{\mathrm{U}}
\newcommand{\BH}{\mathrm{BH}}
\newcommand{\BO}{\mathrm{BO}}

\newcommand{\XX}{\mathbf{X}}

\newcommand{\YY}{\mathbf{Y}}

\newcommand{\e}{\mathrm{e}}

\long\def\/*#1*/{}
\newcommand\prob[1]{\mathbb{P}\left(#1\right)}  
\newcommand\expt[1]{\mathbb{E}\left(#1\right)}  
\newcommand\exptt[1]{\mathbb{E}_t\left(#1\right)}
\newcommand{\pto}{\ensuremath{\xrightarrow{\sss \mathbb{P}}}}  
\newcommand{\dto}{\ensuremath{ \xrightarrow{\sss d}}}  
\newcommand{\dif}{\ensuremath{\mathrm{d}}} 

\newcommand\disteq{\stackrel{\text{\scriptsize d}}{=\joinrel=}} 
\newcommand{\R}{\mathbb{R}}                 
\newcommand{\PP}{\mathbb{P}}  				
\newcommand{\floor}[1]{\ensuremath{\left\lfloor #1 \right\rfloor}}
\newcommand{\bld}[1]{\ensuremath{\boldsymbol{#1}}}

\newcommand{\var}[1]{\ensuremath{\mathrm{Var}\left(#1\right)}}
\newcommand{\oP}{\ensuremath{o_{\sss\PP}}}
\newcommand{\OP}{\ensuremath{O_{\sss\PP}}}
\newcommand{\E}{\mathbb{E}}

\newcommand{\CRG}{\text{\textsc{\textsc{crg}}} }
\newcommand{\RGG}{\text{\textsc{\textsc{rgg}}} }
\newcommand{\RSA}{\text{\textsc{\textsc{rsa}}} }

\begin{document}
\title{Corrected mean-field model for random sequential \\adsorption on random geometric graphs}

\author{Souvik Dhara}
\author{Johan S.H.~van Leeuwaarden}
\author{Debankur Mukherjee}
\affil{Department of Mathematics and Computer Science,\\
Eindhoven University of Technology, The Netherlands}

\renewcommand\Authands{, and }

\date{\today}
\maketitle

\numberwithin{equation}{section}
\numberwithin{theorem}{section}

\begin{abstract}
A notorious problem in mathematics and physics is to create a solvable model for random sequential adsorption of non-overlapping congruent spheres in the $d$-dimensional Euclidean space with $d\geq 2$. 
Spheres arrive sequentially at uniformly chosen locations in space and are accepted only when there is no overlap with previously deposited spheres. 
Due to spatial correlations, characterizing the fraction of accepted spheres remains largely intractable. 
We study this fraction by taking a novel approach that compares random sequential adsorption in Euclidean space to the nearest-neighbor blocking on a sequence of \emph{clustered random graphs}.
This random network model can be thought of as a \emph{corrected mean-field model} for the interaction graph between the attempted spheres.
Using functional limit theorems, we characterize the fraction of accepted spheres and its fluctuations. \\

\end{abstract}

\maketitle


\noindent

\section{Introduction}

Random sequential adsorption of congruent spheres in the $d$-dimensional Euclidean space has been a topic of great interest across the sciences, serving as basic models in  condensed matter and quantum physics~\cite{SBVK2014,LFCDJCZ2001,JCZRCL2000,SWM2010}, nanotechnology \cite{CKE2002,DGPLCM2002}, information theory and optimization problems~\cite{S2001,GMS2012,KT2013}.
Random sequential adsorption also arises naturally in experimental settings, ranging from the deposition of nano-scale particles on polymer surfaces, adsorption of proteins on solid surfaces to the creation of logic gates for quantum computing, and
many more applications in domains as diverse biology, ecology and sociology, see~\cite{CAP07,TS2010,T2013} for extensive surveys.
We refer with random sequential adsorption ($\RSA$) to the dynamic process defined as follows:   
At each time epoch, a point appears at a uniformly chosen location in space, and an attempt is made to place a sphere of radius $r$ with the chosen point as its center. 
The new sphere must either fit in the empty area without overlap with the spheres deposited earlier, or its  deposition attempt is discarded.
After $n$ deposition attempts, the quantity of interest is the proportion of accepted spheres, or equivalently, the volume covered by the accepted spheres. Fig.~\ref{fig:sfig1} illustrates an instance of this \RSA process in 2D.
\begin{figure}
\begin{subfigure}{.5\textwidth}
  \centering
  \includegraphics[scale=.45]{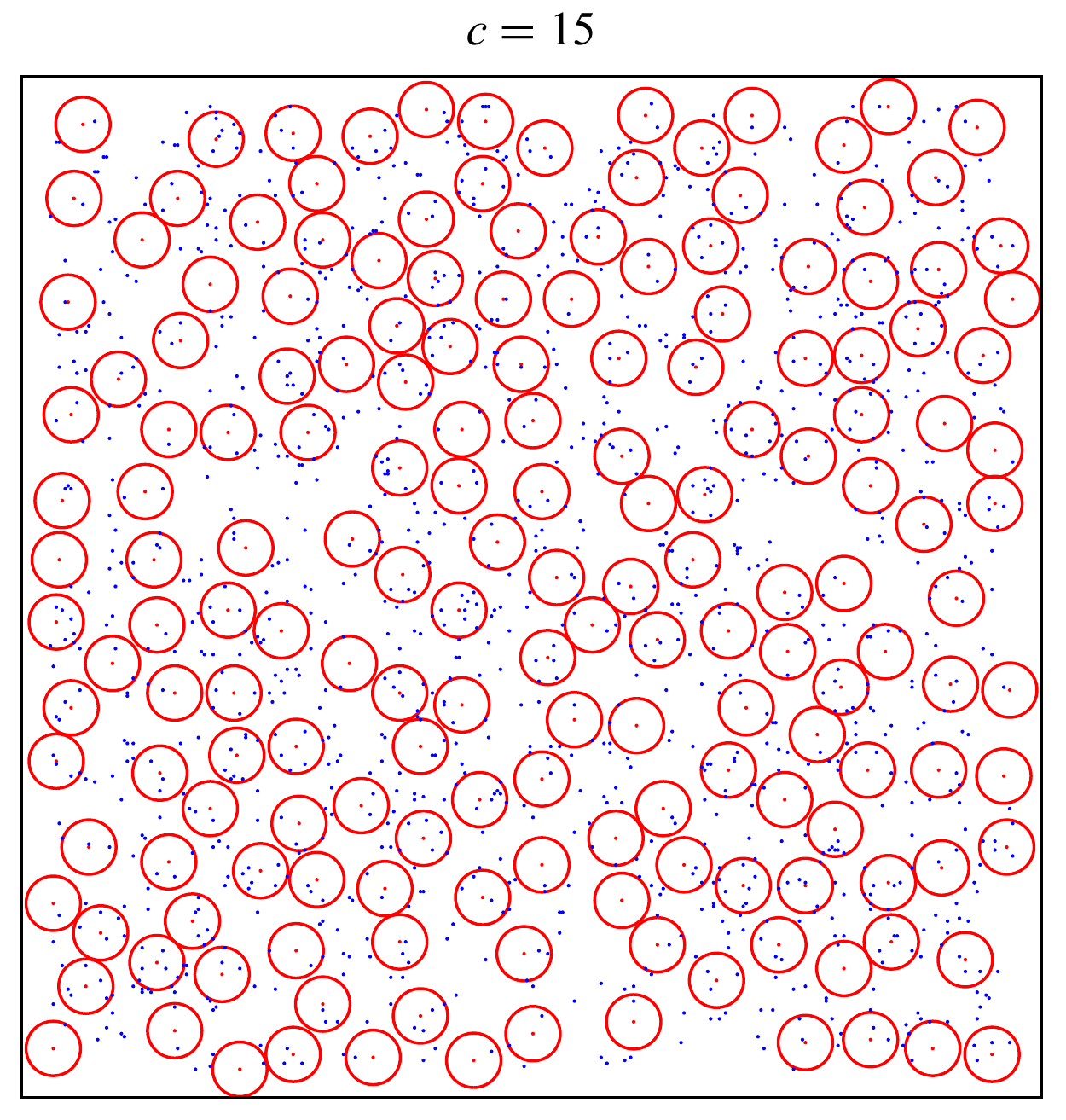}
  \caption{\RSA in 2D}
  \label{fig:sfig1}
\end{subfigure}%
\begin{subfigure}{.5\textwidth}
  \centering
  \includegraphics[scale=.45]{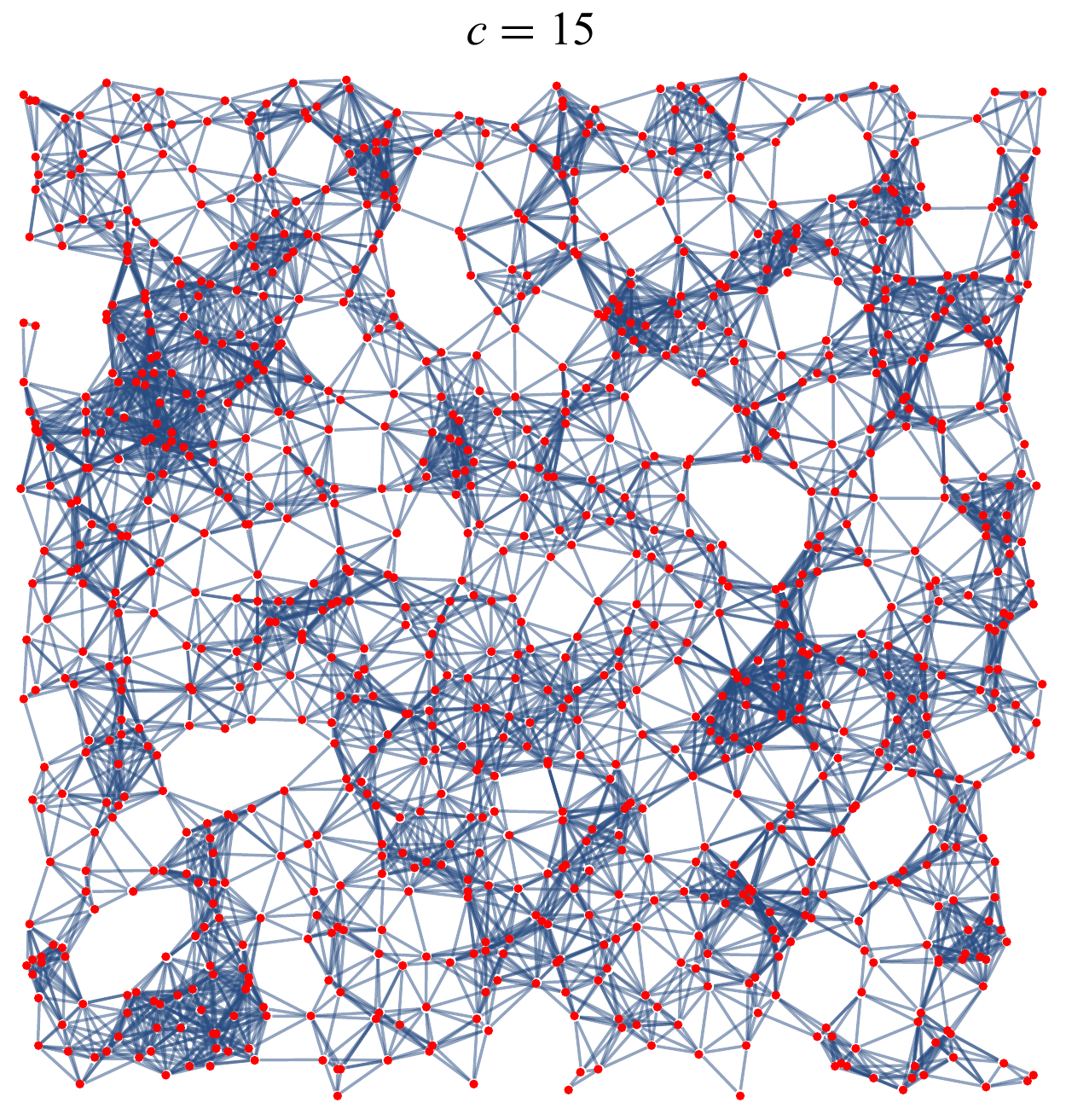}
  \caption{\RGG network}
\label{fig:clique}
\end{subfigure}
\caption{(a) Random sequential adsorption in 2D with density $c=15$.
Dots indicate the centers of accepted (red) and discarded (blue) spheres.
(b) $\RGG(15,2)$ graph with 1000 vertices: Two vertices share an edge if they are less than $2r$ distance apart, where $r$ is such that a vertex has on average $c=15$ neighbors. Notice the many local clusters. 
}\label{fig:rsa}
\end{figure}

Equivalently, we may think of the interaction network of the $n$ chosen centers of spheres by drawing an edge between two points if they are at most $2r$ distance apart. 
This is because a deposition attempt can block another deposition attempt if and only if the centers are at most $2r$ distance apart.
The obtained random graph is known as the random geometric graph ($\RGG$) \cite{P2003}.
The fraction of accepted spheres can be obtained via the following greedy algorithm to find independent sets of $\RGG$: Given a graph $G$, initially, all the vertices are declared inactive.  Sequentially activate uniformly chosen inactive vertices of the graph and block the neighborhood until all the inactive vertices are exhausted. 
We refer to the above greedy algorithm as $\RSA$ on the graph $G$.
If $G$ has the same distribution as $\RGG$ on $n$ vertices, then the final set of active vertices has the same distribution as the number of accepted spheres in the continuum after $n$ deposition attempts.
Thus, we one can equivalently study $\RSA$ on $\RGG$ to obtain the fraction of accepted spheres when $\RSA$ is applied in continuum. 

The precise setting in this paper considers $\RSA$
in a finite-volume box $[0,1]^d$ with periodic boundary, filled with `small' spheres of radius $r$ and volume $V_d(r)$ \cite{PY02,SBVK2014,SJK15}. 
Since the volume of $[0,1]^d$ is~1, the probability that two randomly chosen vertices share an edge in the interaction network is equal to  the  volume  of a  sphere  of  radius $2r$ given by 
$V_d(2r)=\pi^{d/2} (2r)^d/\Gamma(1+d/2)
$.
Thus the average vertex degree in the $\RGG$  is $c=n V_d(2r)$, and
since $c$ is also the average number of overlaps per sphere, with all other attempted spheres, we interchangeably use the terms density and average degree for $c$.
We operate in the sparse regime, where both $r\to 0$ and $n\to \infty$, so that $c\geq 0$ is an arbitrary but fixed constant.
In fact, maintaining a constant density $c$ in the large-network limit is necessary to observe a non-degenerate limit of the fraction of accepted spheres.
In other words, as we will see, the jamming fraction converges to 1 or 0 when $c$ converges to 0 or infinity. 
Thus, in order for $c$ to remain fixed as $n\to\infty$, the radius  should scale as a function of $n$ according to
\begin{equation}\label{radius}
r=r(n)=\frac{1}{2}\left[\frac{c \Gamma(1+d/2)}{n\pi^{d/2}}\right]^{1/d}.
\end{equation}
Notice that it is equivalent to consider the deposition of spheres with fixed radii into a box of growing volume.
We parameterize the $\RGG$ model by the density $c$ and the dimension $d$, and henceforth write this as $\RGG(c,d)$.
A typical instance of $\RGG(5,2)$ with $n = 1000$ vertices is shown in Fig.~\ref{fig:clique}.
Let $J_n(c,d)$ be the fraction of active vertices in the $\RGG(c,d)$ model on $n$ vertices.
While it was proved in~\cite{PY02} that $\lim_{n\to\infty}J_n(c,d)= J(c,d)$ exists, no quantitative characterization of $J(c,d)$ for dimensions $\geq 2$ has been provided till date, and 
so far the main methods to study this problem rely on extensive simulations~\cite{DC02,B11,S1967,TUS2006,ZT2013,
TST1991,F1980}.

In this paper we  propose a novel approach for the study of the fraction of accepted spheres that considers $\RSA$ on a clustered random graph model, designed to match the local spatial properties of the $\RGG$ model in terms of average degree and clustering.
Contrary to the $\RGG$ model, the proposed random graph model is amenable to rigorous mathematical treatment, including exact analysis of the limiting jamming fraction and its fluctuations. 

The paper is structured as follows: Section~\ref{sec:CRG} introduces the clustered random graph and the correspondence with the random geometric graph. Section~\ref{sec:results} presents the main results for the jamming fraction in the mean-field regime. 
We also show through extensive simulations that the mean-field approximations are accurate for all densities and dimensions. 
Sections~\ref{sec:proof}--\ref{sec:clustering-coeff} contain all the proofs, and we conclude in Section~\ref{sec:discussion} with a discussion.

\section{Clustered random graphs}\label{sec:CRG}
\begin{figure}
\centering
\includegraphics[scale=.6]{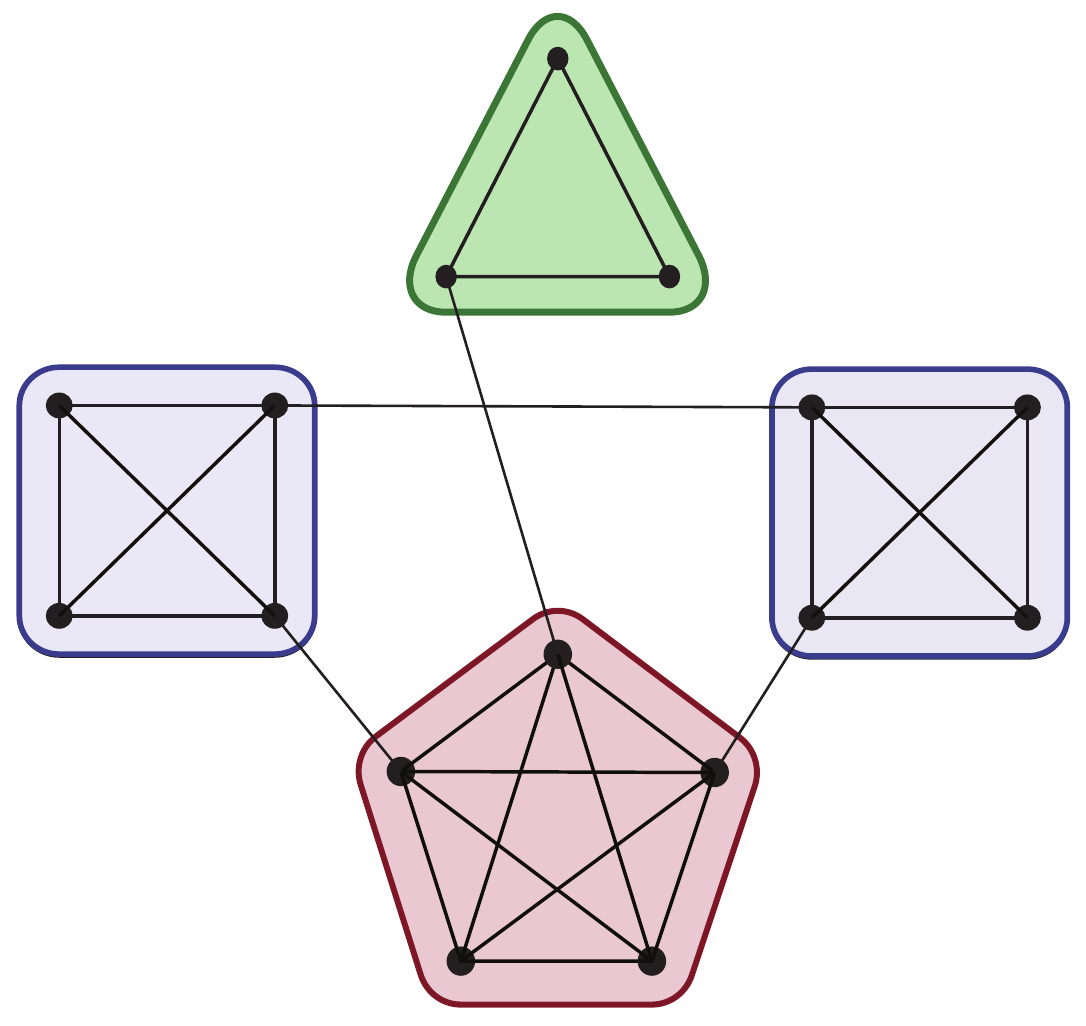}
\caption{Example topology generated by the $\CRG(c,\alpha)$ model.}
\label{fig:crg}
\end{figure}
Random graphs serve to model large networked systems, but are typically unfit for capturing local clustering in the form of relatively many short cycles. This can be resolved by locally adding so-called households or small dense graphs~\cite{BBS13,CL14,T07,N09,KN10,SVV16,SVV16SR,HLS15}.
Vertices in a household have a much denser connectivity to all (or many) other household members, which enforces local clustering.
We now introduce a specific household model, called clustered random graph model ($\CRG$), designed for the purpose of analyzing the $\RSA$ problem.
An arbitrary vertex in the \CRG model has local or short-distance connections with nearby vertices, and global or long-distance connections with the other vertices. When pairing vertices, the local and global connections are formed according to different statistical rules.
The degree distribution of a typical vertex is taken to be Poisson$(c)$ (approximately) in both the $\RGG$ and $\CRG$ model. 
Thus a typical vertex, when activated, blocks approximately Poisson$(c)$ other vertices.
In the $\CRG(c,\alpha)$ model however, the total mass of connectivity measured in the density parameter $c$, is split into $\alpha c$ to account for direct local blocking and $(1-\alpha)c$ to incorporate the propagation of spatial correlations over longer distances.
The $\CRG(c,\alpha)$ model with $n$ vertices is then defined as follows (see Fig.~\ref{fig:crg}):
\begin{itemize}
\item Partition the $n$ vertices into random households of size $1+\mathrm{Poisson}(\alpha c)$. This can be done by sequentially selecting
$1+\mathrm{Poisson}(\alpha c)$ vertices uniformly at random and declaring them as a household, and repeat this procedure until at some point
the next $1+\mathrm{Poisson}(\alpha c)$ random variable is  at most the number of remaining vertices. 
All the remaining vertices are then declared a household too, and the household formation process is completed.
\item Now that all vertices are declared members of some household, the random graph is constructed according to a local and a global rule.
The local rule says that all vertices in the same household get connected by an edge, leading to complete graphs of size 1+Poisson($\alpha c$).
The global rule adds a connection between
 any two vertices belonging to two different households with probability $(1-\alpha)c/n$.
 \end{itemize}
This creates a class of random networks with average degree $c$ and tunable level of clustering via the free parameter $\alpha$.
With the goal to design a solvable model for the $\RSA$ process, the $\CRG(c,\alpha)$ model
has $nc/2$ connections to build a random structure that mimics the local spatial structure of the $\RGG(c,d)$ model on $n$ vertices.

Seen as the topology underlying the $\RSA$ problem, the $\CRG(c,\alpha)$ model incorporates local clusters of overlapping spheres, which occur naturally in random geometric graphs;
 see Fig.~\ref{fig:clique}.
We can now also consider \RSA on the $\CRG(c,\alpha)$ model, by using the greedy algorithm that constructs an independent set on the graph by sequentially selecting vertices uniformly at random, and placing them in the independent set unless they are adjacent to some vertex already chosen. 
The jamming fraction $J^\star_n(c,\alpha)$ is then the size of the greedy independent set divided by the network size $n$.
From a high-level perspective, we will solve the \RSA problem on the $\CRG(c,\alpha)$ model, and translate this solution into an equivalent result for $\RSA$ on the $\RGG(c,d)$. 

Our ansatz is that for large enough $n$, a unique relation can be established between dimension $d$ in $\RGG$ and the parameter $\alpha = \alpha_d$ in $\CRG$, so that the jamming fractions are comparable, i.e.,~$J_n(c, d)\approx J^\star_n(c, \alpha_d),$
and virtually indistinguishable in the large network limit.
In order to do so, we map the $\CRG(c,\alpha)$ model onto the $\RGG(c,d)$  model by imposing two natural conditions.
The first condition matches the average degrees in both topologies, i.e., $c$ is chosen to be equal to $nV_d(2r)$.
The second condition tunes the local clustering. 

Let us first describe the clustering in the $\RGG$ model.
Consider two points chosen uniformly at random in a $d$-dimensional hypersphere of radius $2r$. 
Then what is the probability that these two points are themselves at most $2r$ distance apart?  
From the $\RGG$ perspective, this corresponds to the probability that, conditional on two vertices $u$ and $v$ being neighbors, a uniformly chosen neighbor $w$ of $u$ is also a neighbor of $v$, which is known as the local clustering coefficient \cite{N2010}. 
In the $\CRG(c,\alpha)$ model, on the other hand, the relevant measure of clustering is $\alpha$, the probability that a randomly chosen neighbor is a neighbor of one of its household members. 
 We then choose the unique $\alpha$-value that equates to the clustering  coefficient of $\RGG$.
Denote this unique value by $\alpha_d$, to express its dependence on the dimension~$d$.
In Section~\ref{sec:clustering-coeff}, we show that
\begin{eq}\label{eq:alpha-choice}
  \alpha_d = d\int_{0}^1x^{d-1}I_{1-\frac{x^2}{4}}\Big(\frac{d+1}{2}, \frac{1}{2}\Big)\dif x
\end{eq}
with $I_z(a,b)$ the normalized incomplete beta integral.
Table~\ref{tab1} shows the numerical values of $\alpha_d$ for dimensions 1 to 5.
\begin{table}
\centering
\begin{tabular}{C{.8cm}|C{1.5cm}|C{1.5cm}|C{1.5cm}|C{1.5cm}|C{1.5cm}}
\hline
{\rm $d$} &1 & 2 & 3 & 4 & 5 \\
\hline
$\alpha_d$ &0.750000&  0.586503 & 0.468750 & 0.379755 & 0.310547\\
\hline
\end{tabular}%
\caption{$\alpha_d$ for dimensions 1 to 5.}
\label{tab1}
\end{table}
With the uniquely characterized $\alpha_d$ in \eqref {eq:alpha-choice}, the $\CRG(c,\alpha_d)$ model can now serve as a generator of random topologies for guiding the \RSA process.

In contrast to the Euclidean space, $\RSA$ on the $\CRG(c,\alpha_d)$ model  is analytically solvable, even at later times when the filled space becomes more dense (large $c$). 
To do so, we will extend the mean-field techniques recently developed for analyzing \RSA on random graph models~\cite{DLM16,BJL15,BJM13,SJK15}.
The main goal of these works was to find greedy independent sets (or colorings) of large random networks.
All these results, however, were obtained for non-geometric random graphs, typically used as first approximations for sparse interaction networks in the absence of any known geometry.

\section{Main results}\label{sec:results}

\subsection{Limiting jamming fraction}
For the $\CRG(c,\alpha)$ model on $n$ vertices, recall that $J^\star_n(c,\alpha)$ denotes the fraction of active vertices at the end of the $\RSA$ process.
We then have the following result, which characterize the limiting fraction:


\begin{theorem}[Limiting jamming fraction]\label{th:fluid}
 For any $c>0$ and $\alpha\in [0,1]$, as $n\to\infty$, $J^\star_n(c,\alpha)$ converges in probability to $J^\star(c,\alpha)$, where $J^\star(c,\alpha)$ is the smallest nonnegative root of the deterministic function $x(t)$ described by the integral equation
 \begin{equation}\label{diffeq}
 x(t)=1-t-\int_0^t\left(\frac{x(s)\alpha c}{1-(\alpha c+1) s}+(1-\alpha)cx(s)\right)\dif s.
\end{equation}
\end{theorem}
The ODE~\eqref{diffeq} can be understood intuitively in terms of the algorithmic description in Section~\ref{sec:algo} that sequentially explores the graph while activating the allowed vertices.
Rescale time by $n$, so that after rescaling the algorithm has to end before time $t=1$ (because the network size is $n$). Now think of $x(t)$ as the fraction of neutral vertices at time $t$.
Then clearly $x(0) = 1$, and the drift $-t$ says that one vertex activates per time unit.
Upon activation, a vertex on average blocks its $\alpha c$ household members and  $(1-\alpha)c$ other vertices outside its household.
 At time $t$, the fraction of vertices that are not members of any discovered households equals on average $(1 - (1+\alpha c)t)$ and all vertices which are not part of any discovered households, are potential household members of the newly active vertex (irrespective of whether it is blocked or not). 
 Since household members are uniformly selected at random, only a fraction $x(t) / (1 - (1+\alpha c)t)$ of the new $\alpha c$ household members will belong to the set of neutral vertices. 
 Moreover, since all $x(t)n$ vertices are being blocked by the newly active vertex with probability $(1-\alpha)c/n$, on average $(1-\alpha)c x(t)$ neutral  vertices will be blocked due to distant connections.
  Notice that the graph will be maximally packed when $x(t)$ becomes zero, i.e., there are no neutral vertices that can become active. This explains why $J^\star(c,\alpha)$ should be the time $t$ when $x(t)=0$, i.e., smallest root of~\eqref{diffeq}.
  
\begin{figure}
\centering
\includegraphics[scale=0.6]{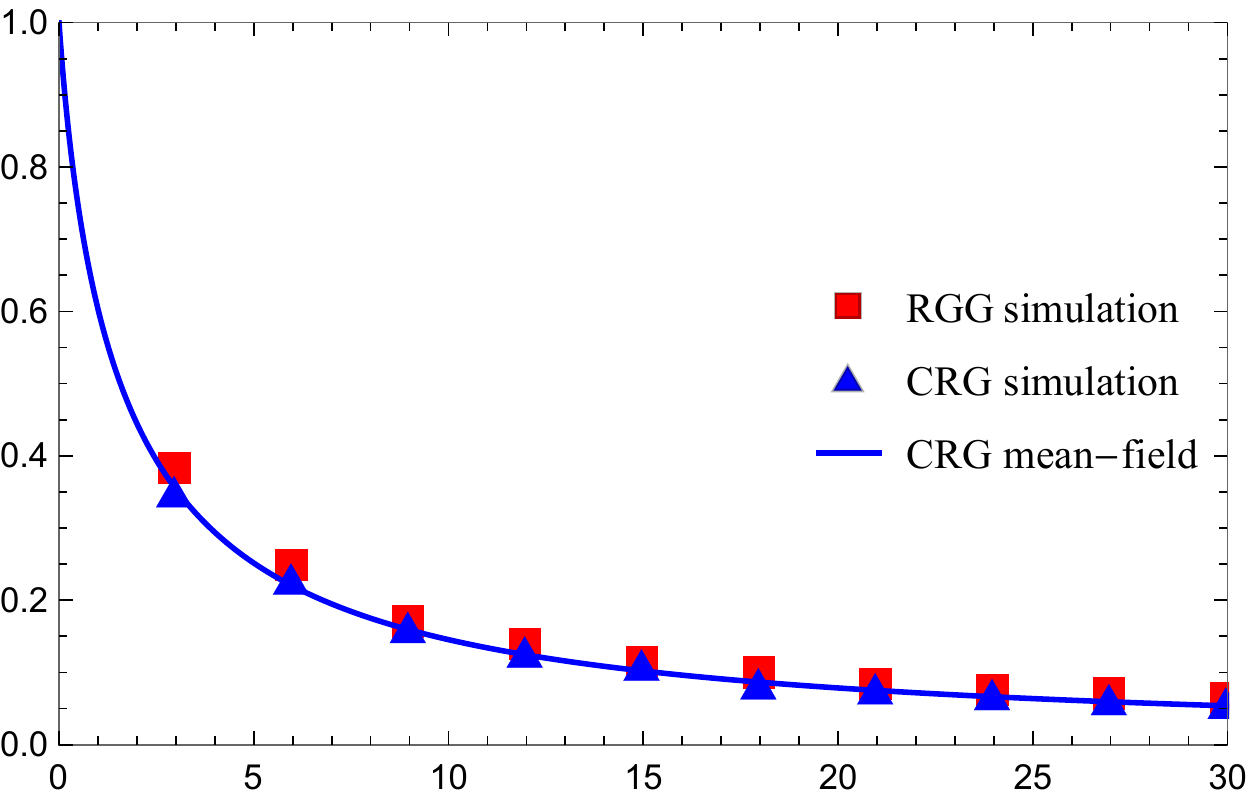}
\caption{Validation of the mean-field limit $J^\star(c,\alpha_2)$ with the simulation results from $\CRG(c,\alpha_2)$, and $\RGG(c,2)$ with 1000 vertices for $0\leq c\leq 30$.}
\label{fig:crg-rgg}
\end{figure}
Upon substituting $\alpha=\alpha_d$, $J^\star(c,\alpha_d)=\lim_{n\to\infty}J^\star_n(c,\alpha_d)$ is completely characterized by \eqref{diffeq} and serves an approximation for the intractable counterpart $J(c,d)$, the limiting jammed fraction for the $\RGG(c,d)$ model.
The choice of $\alpha_d$, as discussed earlier, is given by \eqref{eq:alpha-choice} and shown in Table~\ref{tab1}.
Fig.~\ref{fig:crg-rgg} validates the mean-field limit for the \CRG model, and shows the theoretical values $J^\star(c,\alpha_2)$  from Theorem~\ref{th:fluid}, along with the simulated values of $J_n(c,2)$ on the $\RGG(c,2)$ model for values of $c$ ranging from 0~to~30.
Fig.~\ref{fig:dimension-one-to-five} shows further comparisons between $J^\star(c,\alpha_d)$ and $J_n(c,d)$ for dimensions $d=3,4,5$, and densities $0\leq c\leq 30$.
All simulations use $n=1000$ vertices.
\begin{figure}
\begin{subfigure}{.33\textwidth}
  \centering
  \includegraphics[scale=.4]{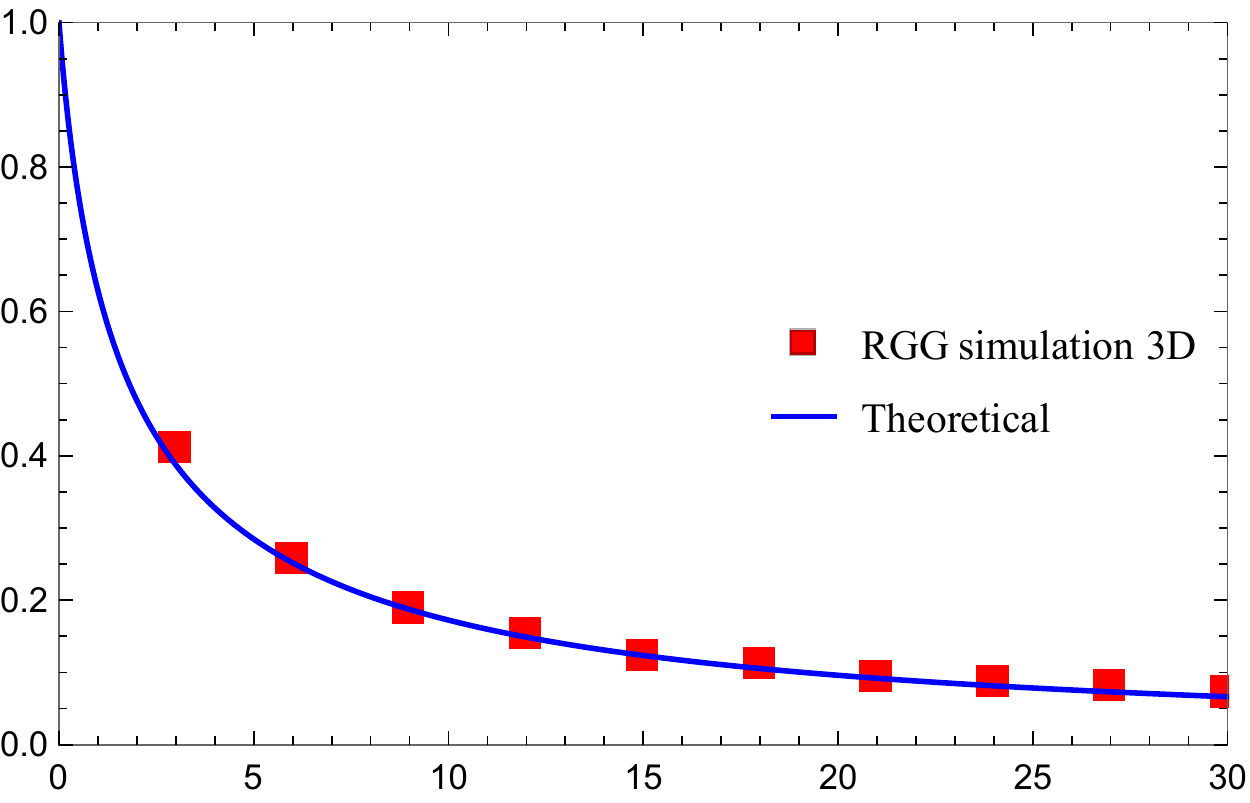}
  \caption{3D}
  \label{fig:3d}
\end{subfigure}%
\begin{subfigure}{.33\textwidth}
  \centering
  \includegraphics[scale=.4]{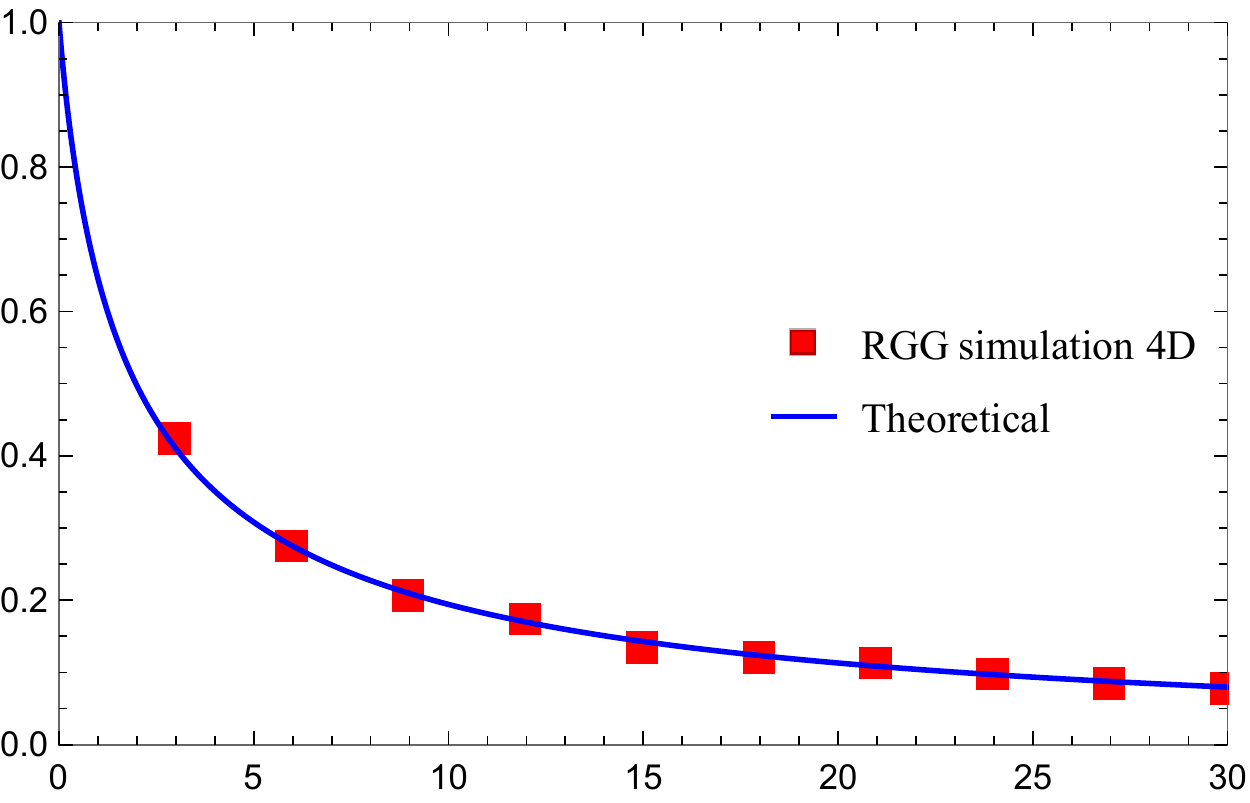}
  \caption{4D}
\label{fig:4d}
\end{subfigure}
\begin{subfigure}{.33\textwidth}
  \centering
  \includegraphics[scale=.4]{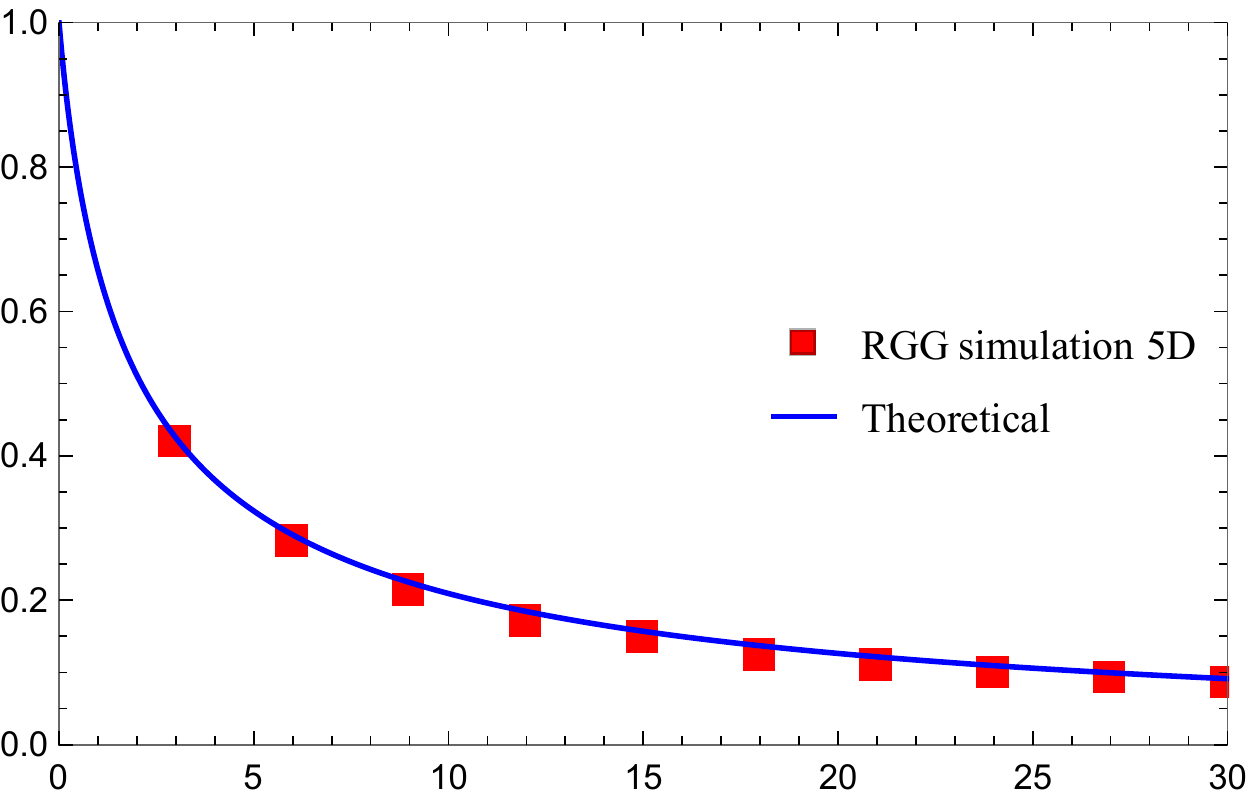}
  \caption{5D}
\label{fig:5d}
\end{subfigure}
\caption{Simulation with 1000 vertices of $\RGG(c,d)$ and the value of $J^\star(c,\alpha_d)$   for $0\leq c\leq 30$ and $d=3,4,5$. 
\label{fig:dimension-one-to-five}
}
\end{figure}
The remarkable agreement of the $J^\star(c,\alpha_d)$-curves with the simulated results across all dimensions shows that the integral equation \eqref{diffeq} accurately describes the mean-field large-network behavior of the $\RSA$ process, not only for the \CRG model, but also for the $\RGG$ model.
The following result is a direct consequence of Theorem~\ref{th:fluid}, and gives a simple law to describe the asymptotic fraction $J^\star(c,\alpha_d)$ in the large density ($c\to\infty$) regime.
\begin{corollary}\label{cor:J-large-c} As  $c\to\infty$, $
J^\star(c,\alpha_d) \sim (1+\alpha_dc)^{-1}.
$
\end{corollary}
Hence, for large enough $c$, $J^\star(c,\alpha_d)\approx(1+\alpha_d c)^{-1}$ serves as an approximation for all dimensions.
Due to the accurate prediction provided by the \CRG model, the total scaled volume $c{J}(c,d)/2^d$ covered by the deposited
spheres  in dimension $d$ can be well approximated.
Indeed, for large $c$, Corollary~\ref{cor:J-large-c}  yields $J^\star(c,\alpha_d)\sim 1/(\alpha_dc)$, and in any dimension~$d$, our model leads to a precise characterization of the covered volume given by
\begin{equation}
J^\star(c,\alpha_d)\times\frac{c}{2^d} = \frac{1}{2^d\alpha_d}.
\end{equation}

Notice that $\alpha_d\to 0$ as $d\to\infty$.
Thus, the interaction network described by the $\CRG(c,\alpha_d)$ model becomes almost like the (pure) mean-field Erd\H{o}s-R\'enyi random graph model, which supports the widely believed conjecture that in high dimensions the interaction network associated with the random geometric graph loses its local clustering property \cite{DGLU11}.
\subsection{Fluctuations of the jamming fraction}
The next theorem characterizes the fluctuations of $J_n^\star(c,\alpha)$ around its mean:
\begin{theorem}[{CLT for jamming fraction}]
\label{th:jam-diffusion}
As $n\to\infty$,
$$\sqrt{n}(J_n^\star(c,\alpha)-J^\star(c,\alpha))\dto Z,$$
where $Z$ has a normal distribution with mean zero and variance $V^\star(c,\alpha)$. Here  $J^\star(c,\alpha)$ is given by Theorem~\ref{th:fluid},
 and $V^\star(c,\alpha)=\sigma_{xx}(J^\star(c,\alpha))$ with $\sigma_{xx}(t)$ being the unique solution of the  system of differential equations, for $0\leq t<1/\mu$,\\
 \begin{eq}\label{defn:functions0}
  \frac{\dif \sigma_{xx}(t)}{\dif t}&=2\sigma_{xx}(t)f(t)+2\sigma_{xy}(t)g(t)+\beta(t),\\  
  \frac{\dif \sigma_{xy}(t)}{\dif t}&=\sigma_{xy}(t)f(t)+tg(t)\sigma^2+\sqrt{\beta(t)}\sigma\rho(t)
\end{eq}
with
\begin{eq}\label{defn:functions}
y(t)&=1-\mu t,\quad  f(t)=-\frac{\mu-1}{y(t)}-\lambda, \quad  g(t)=\frac{(\mu-1)x(t)}{y(t)^2},\\
 \beta(t)&=\left[\frac{(\mu-1)}{y(t)}+\lambda\right]x(t),\quad
  \rho(t)=\frac{\sigma}{\sqrt{\beta(t)}}\frac{x(t)}{y(t)}.
\end{eq}
\end{theorem}
%
Fig.~\ref{histogram} confirms that the asymptotic analytical variance given in~\eqref{defn:functions0} and~\eqref{defn:functions} is a sharp approximation for the \CRG model with only 2000 vertices.
\begin{figure}
\begin{subfigure}{.5\textwidth}
  \centering
  \includegraphics[scale=.5]{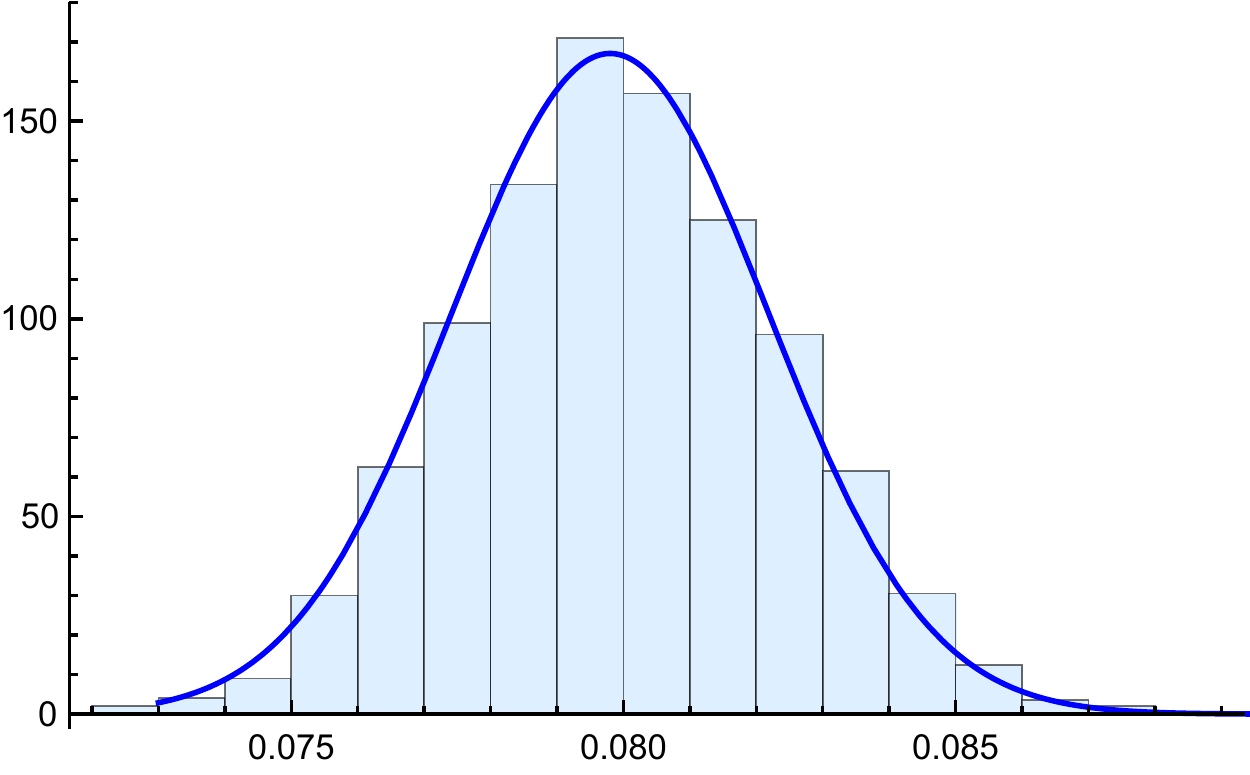}
  \caption{Fitted normal curve}
  \label{histogram}
\end{subfigure}%
\begin{subfigure}{.5\textwidth}
  \centering
  \includegraphics[scale=.5]{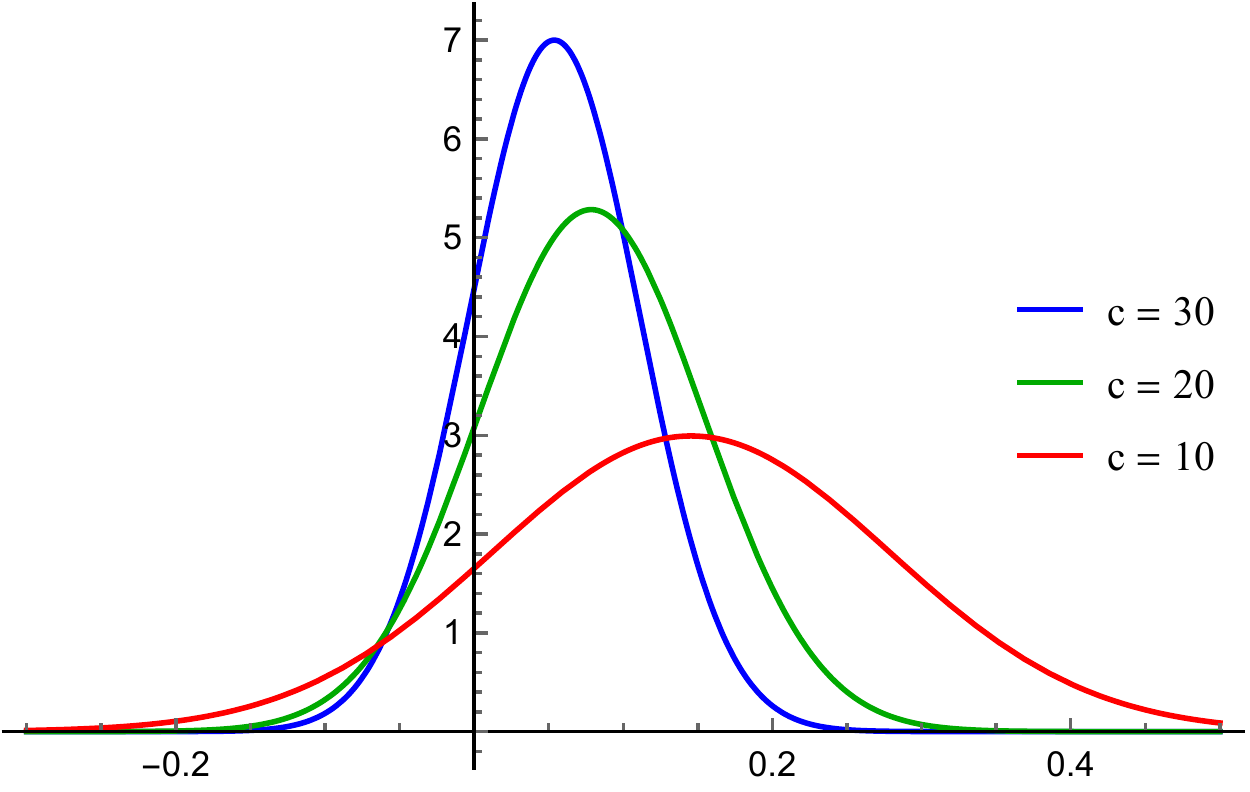}
  \caption{Effect of density on variance}
\label{figguass1}
\end{subfigure}
\caption{(a) Fitted normal curve for 2000 repetitions of the $\CRG(20,\alpha_2)$ model with 1000 vertices. The solid curve represents the normal density with properly scaled theoretical variance $V^\star(c,\alpha_2)$, centered around the sample mean.
(b)  Fitted normal curves for the $\CRG(c,\alpha_2)$ model for increasing $c$ values 10, 20, and 30. As $c$ increases, the curve become more sub-Poissonian.
}
\end{figure}
Table~\ref{tab2} shows numerical values of $V^\star(c,\alpha_d)$
 and compares the analytically obtained values of $J^\star(c,\alpha_d)$ and $V^\star(c,\alpha_d)$, and simulated  mean and variance for the random geometric graph ensemble. The agreement again confirms the appropriateness of the $\CRG(c,\alpha_d)$ model for modeling the continuum $\RSA$. Furthermore, $V^\star(c,\alpha_d)$ serves as an approximation for the value of $V(c,d)$, the asymptotic variance of $J(c,d)$ (suitably rescaled).
\begin{table}
\centering
\begin{tabular}{|C{1cm}|C{1cm}|C{1.75cm}|C{1.75cm}|C{1.75cm}|C{1.75cm}|}
\cline{3-6}
\multicolumn{2}{c|}{} &\multicolumn{2}{c}{\RGG}&\multicolumn{2}{|c|}{\CRG}\\
\hline
\multicolumn{1}{|c|}{$n$} & $c$ & $J_n(c,2)$ & $V_n(c,2)$ & $J^\star(c,\alpha_2)$ & $V^\star(c,\alpha_2)$\\
\hline
200 & 10   &  0.1618    &     0.0166     &       &         \\
500 & 10   &   0.1608   &     0.0158     &   0.1454   &    0.0178       \\
1000 & 10  & 0.1623  & 0.0155  &      &        \\
\hline
200 & 20  &  0.0887    &     0.0062     &      &         \\
500 & 20  &   0.0892  &    0.0068      &   0.0786   &    0.0057 \\
1000 & 20  &  0.0890    &    0.0067     &      &         \\
\hline
200 & 30  &  0.0619    &     0.0039     &      &         \\
500 & 30  &   0.0620   &    0.0041      &   0.0538   &    0.0032 \\
1000 & 30  &  0.0615    &     0.0043     &      &         \\
\hline
\end{tabular}%
\caption{Comparison between the observed mean and scaled variance $n\mathrm{Var}(J_n(c,2))$ for the $\RGG$ model, and the theoretical mean and variance from Theorem~\ref{th:jam-diffusion} in dimension 2.
The sample means and variances for the $\RGG$ model are calculated over 150 samples.}
\label{tab2}
\end{table}
Fig.~\ref{figguass1} shows the density function of the random variable based on the Gaussian approximation in Theorem~\ref{th:jam-diffusion}. 
We observe that both the mean and the fluctuations around the mean decrease with $c$. 
Indeed, the variance-to-mean ratio has been typically observed to be smaller than one for \RSA in the continuum, and it is generally believed that the jamming fractions are typically of sub-Poissonian nature with fluctuations that are not as large as for a Poisson distribution; 
see for instance the Mandel Q parameter in quantum physics \cite{SJK15}. 
So, while a closed-form expression remains out of reach (as for the Mandel Q parameter \cite{SJK15}), our solvable model gives a way to  describe approximately the variance-to-mean ratio as $V^\star(c,\alpha_d)/J^\star(c,\alpha_d)$.

\section{Proof of Theorem~\texorpdfstring{\ref{th:fluid}}{3.1}}\label{sec:proof}
In this section we analyze several asymptotic properties of \RSA on the $\CRG(c,\alpha)$ model.
In particular, we will prove Theorem~\ref{th:fluid}.
We first introduce an algorithm that sequentially activates the vertices while obeying the hard-core exclusion constraint, and then analyze the exploration algorithm (see~\cite{BJS15, DLM16, BJL15} for similar analyses in various other contexts).
The idea is to keep track of the number of vertices that are not neighbors of already actives vertices (termed unexplored vertices), so that when this number becomes zero, no vertex can be activated further.
The number of unexplored vertices can then be decomposed into a drift part which converges to a deterministic function and a fluctuation or martingale part which becomes asymptotically negligible in the mean-filed case (Theorem~\ref{th:fluid}) but gives  rise to the a system of SDEs with variance \eqref{defn:functions0}.
The proof crucially relies on the Functional Laws of Large Numbers (FLLN) and the Functional Central Limit Theorem (FCLT).
The key challenge here is that the process that keeps track of the number of unexplored vertices while the exploration algorithm is running does not yield a Markov process, so we have to introduce another process to make the system Markovian and analyze this two-dimensional system. 

For each vertex, the neighboring vertices inside and outside its own household will be referred to as `household  neighbors' and `distant neighbors', respectively.
If $H$ denotes the size of the households, then $H\sim 1+\mathrm{Poisson}(\alpha c)$. Therefore,  $\expt{H}=1+\alpha c$, and $\mathrm{Var}(H)=\alpha c$.
Furthermore, any two vertices belonging to two different households are
connected by an edge with probability $p_n=(1-\alpha)c/n$, so the number of distant neighbors is a Bin$(n-H-1,p_n)$ random variable, Poisson$((1-\alpha)c)$ in the large $n$ limit.
As mentioned earlier, the total number of neighbors, is then asymptotically given by a Poisson$(c)$ random variable.
In this section we fix $c>0$ and $\alpha\in[0,1]$, and simply write $J^\star_n$ and $J^\star$ for $J^\star_n(c,\alpha)$ and $J^\star(c,\alpha)$ respectively.\\

\noindent
{\bf Notation.}
We will use boldfaced letters to denote stochastic processes and vectors.
A sequence of random variables $\{X_n\}_{n\geq 1}$ is said to be $\OP(f(n))$, or $\oP(f(n))$, for some function $f:\R\to\R_+$, if the sequence of scaled random variables $\{X_n/f(n)\}_{n\geq 1}$ is  tight, or converges to zero in probability, respectively.
We denote by $D_E[0,\infty)$ the set of all \emph{c\`adl\`ag} (right continuous left limit exists) functions from $[0,\infty)$ to a complete, separable metric space $E$, endowed with the Skorohod $J_1$ topology, and
by `$\dto$' and `$\pto$', convergence in distribution and in probability, respectively. In particular, if the sample paths of a stochastic process $\XX$ are continuous, we write
$\XX_n=\{X_n(t)\}_{t\geq 0}\xrightarrow{\sss d} \XX=\{X(t)\}_{t\geq 0}$, if for any $T\geq 0$,
\begin{equation}
\sup_{t\in [0,T]}|X_n(t)-X(t)|\pto 0\quad\text{as}\quad n\to\infty.
\end{equation}

\subsection{The exploration algorithm}\label{sec:algo}
Instead of fixing a particular realization of the random graph and then studying \RSA on that given graph, we introduce an algorithm which sequentially \emph{activates} the vertices one-by-one, \emph{explores} the neighborhood of the activated vertices, and simultaneously builds the random graph topology on the activated and  explored vertices.
The joint distribution of the random graph and active vertices obtained this way is same as those obtained by first fixing the random graph and then studying \RSA.
The idea of exploring in the above fashion simplifies the whole analysis, since the evolution of the system can be described recursively in terms of the previous states, as  described below in detail. 

Observe that during the process of sequential activation, until the jamming state is reached, the vertices can be
in either of three states: active, blocked, and unexplored (i.e.~vertices with future potential activation).
Furthermore, there can be two types of blocked vertices: (i) blocked due to activation of some household neighbor,
or
(ii) none of the household neighbors is active, but there is an active distant neighbor.
Therefore, at each time $t\geq 0$, categorize the vertices into four sets:
\begin{itemize}
\item A$(t)$: set of all vertices active.
\item U$(t)$: set of all vertices that are not active and that have not been blocked by any vertex in~A$(t)$.
\item BH$(t)$: set of all vertices that belong to a household of some vertex in A$(t)$.
\item BO$(t)$: set of all vertices that do not belong to a household yet, but are blocked due to connections with some vertex in A$(t)$ as a distant neighbor.
\end{itemize}
Note that $\BH(t)\cup\BO(t)$ constitutes the set of all blocked vertices at time $t$, and $\BH(t)\cap\BO(t)=\emptyset$.
Initially, all vertices are unexplored, so that U$(0)= V(G)$, the set of all $n$ vertices.
At time step $t$,
one vertex $v$ is selected from U$(t-1)$ uniformly at random and
is transferred to A$(t)$, i.e., one unexplored vertex becomes active.

We now explore the neighbors of $v$, which can be of two types: the household neighbors, and the distant neighbors.
Further observe that $v$ can have its household neighbors only from the set
 $\U(t-1)\cup\BO(t-1)\setminus \{v\}$, since each vertex in $\BH(t-1)$ already belongs
 to some household.
 Define
 $$H(t)\sim \min\Big\{\mathrm{Poisson}(\alpha c), |\U(t-1)\cup\BO(t-1)\setminus \{v\}|\Big\},$$
 i.e., draw a Poisson$(\alpha c)$ random variable independently of any other process, and if it is smaller than $|\U(t-1)\cup\BO(t-1)\setminus \{v\}|$, then take it to be the value of $H-1$, and otherwise set $H(t)=|\U(t-1)\cup\BO(t-1)\setminus \{v\}|$.
Select $H(t)$ vertices $\{u_1,u_2,\ldots,u_{H}\}$ at random from all vertices in  $U(t-1)\cup \BO(t-1)\setminus \{v\}$.
These $H(t)$ vertices together form the household containing $v$, and are moved to $\BH(t)$, irrespective of the set they are selected from.
To explore the distant neighbors, select one by one, all the vertices in $\U(t-1)\cup \BO(t-1)\cup \BH(t-1)\setminus\{v,u_1,\ldots,u_{H}\}$,
and for every such selected vertex $\bar{u}$, put an edge between $\bar{u}$ and $v$ with probability $p_n$.
Denote the newly created distant neighbors that belonged to $\U(t-1)$
by $\{\bar{u}_1,\ldots,\bar{u}_d\}$, and move these vertices to $\BO(t)$.
In summary, the exploration algorithm yields the following recursion relations:
\begin{align*}
\A(t) &= \A(t-1)\cup\{v\},\\
\U(t) &= \U(t-1)\setminus\{v, u_1,u_2,\ldots,u_{H},\bar{u}_1,\ldots,\bar{u}_d\},\\
\BH(t)  &= \BH(t-1)\cup \{u_1,u_2,\ldots,u_{H}\},\\
\BO(t)  &= \BO(t-1)\cup \{\bar{u}_1,\ldots,\bar{u}_d\}.
\end{align*}
The algorithm terminates when there is no vertex left in the set U$(t)$ (implying that all vertices are either active or blocked), and outputs the cardinality of A$(t)$ as the number of active vertices in the jammed state.
\subsection{State description and martingale decomposition.}\label{sec:mart-decompose}
Denote for $t\geq 0$,
\begin{align*}
X_n(t):=|\U(t)|, \quad
Y_n(t):=|\U(t)\cup \BO(t)|.
\end{align*}
Observe that $\{(X_n(t),Y_n(t))\}_{t\geq 0}$ is a Markov chain.
At each time step, one new vertex becomes active, so that 
$|\A(t)|=t$, and
the total number of vertices in the jammed state is given by the time step when $X_n(t)$ hits zero, i.e.,
the time step when the exploration algorithm terminates.
Let us now introduce the shorthand notation $\mu = \E[H] = 1+\alpha c$, $\sigma^2 = \var{H}= \alpha c$ and $\lambda = (1-\alpha) c$.

\paragraph*{\bf Dynamics of $X_n$.}
First we make the following observations:
\begin{itemize}
\item $X_n(t)$ decreases by one, when a new vertex $v$ becomes active.
\item The household neighbors of $v$ are selected from $Y_n(t-1)$ vertices, and $X_n(t)$ decreases by an amount of the number of such vertices which are in $\U(t-1)$.
\item $X_n(t)$ decreases by the number of distant neighbors of the newly active vertex that belong to $\U(t-1)$ (since they are transferred to $\BO(t)$).
\end{itemize}
Thus,
\begin{equation}
 X_n(t+1)=X_n(t)-\xi_n(t+1)\quad\text{and}\quad X_n(0)=n
\end{equation}
with
\begin{equation}\label{defn:xi}
 \xi_n(t+1)=1+\eta_1(t+1)+\eta_2(t+1),
\end{equation}
where conditionally on $(X_n(t),Y_n(t))$,
\begin{equation}\label{defn:eta1}
\eta_1(t+1)\sim \text{Hypergeometric}(X_n(t), Y_n(t), H(t)),
\end{equation}
i.e.,~$\eta_1(t+1)$ has a Hypergeometric distribution with favorable outcomes $X_n(t)$, population size $Y_n(t)$, and sample size $H(t)$.
Further, conditionally on $(X_n(t),Y_n(t),\eta_1(t+1))$,
\begin{equation}\label{defn:eta2}
\eta_2(t+1) \sim \mathrm{Bin}\Big(X_n(t)-1-\eta_1(t+1),\ \frac{\lambda}{n}\Big).
\end{equation}
Therefore, the drift function of the $\XX_n$ process satisfies
\begin{equation}\label{eq:drift-X}
\begin{split}
  \expt{\xi_n(t+1)| X_n(t), Y_n(t)}&=1+\frac{X_n(t)(\mu-1)}{Y_n(t)}+\left(X_n(t)-1-\frac{X_n(t)(\mu-1)}{Y_n(t)}\right)\frac{\lambda}{n}\\
 &= 1+ \frac{X_n(t)(\mu-1)}{Y_n(t)}+ \frac{\lambda X_n(t)}{n} +\OP(n^{-1}),
\end{split}
\end{equation}
where, in the last step, we have used the fact that $X_n(t)\leq Y_n(t)$.

\paragraph*{\bf Dynamics of $Y_n$.}
The value of $Y_n$ does not change due to the creation of distant neighbors.
At time $t$, it can only decrease due to an activation
of a vertex $v$ (since it is moved to $\A(t)$), and the formation of a household, since all the vertices that make the
household of $v$, were in $\U(t-1)\cup\BO(t-1)$, and are moved to $\BH(t)$.
Thus, at each time step, $Y_n(t)$ decreases on average by an amount $\mu=1+\alpha c$, 
the expected household size, except at the final step when the residual number of vertices can be smaller than the household size.
But this will not affect our asymptotic results in any way, and we will ignore it.
Hence,
\begin{equation}
 Y_n(t+1)=Y_n(t)-\zeta_n(t+1) \quad\text{and}\quad Y_n(0)=n,
\end{equation}where
\begin{equation}
 \expt{\zeta_n(t+1)|X_n(t), Y_n(t)}=\mu.
\end{equation}

\paragraph*{\bf Martingale decomposition.}
Using the Doob-Meyer decomposition \cite[Theorem 4.10]{KS91} of~$\bld{X}_n$, \eqref{eq:drift-X} yields the following martingale decomposition
\begin{equation*}
\begin{split}
 X_n(t)
 &=n-\sum_{i=1}^t\xi_n(i)=n+M_{n}^{\sss X}(t)-t-\sum_{i=1}^t\bigg[\frac{X_n(i-1)(\mu-1)}{Y_n(i-1)}
+ \frac{\lambda X_n(i-1)}{n} +\OP(n^{-1})\bigg],
 \end{split}
\end{equation*}
where $\bld{M}_n^{\sss X}=\{M_n^{\sss X}(t)\}_{t\geq 1}$ is a square-integrable martingale with respect to the usual filtration generated by the exploration algorithm.
Let us now define the scaled processes
$$x_n(t):=\frac{X_n(\floor{nt})}{n}\quad\text{and}\quad y_n(t):=\frac{Y_n(\floor{nt})}{n}.$$
Also define
\begin{equation}\label{eq:delta}
\delta(x,y):=(\mu-1) \frac{x}{y}+\lambda x,\quad\text{for}\quad 0\leq x\leq y,\ y>0.
\end{equation}
Thus, we can write
\begin{equation}\label{eq:mart-decompose}
 \begin{split}
  x_n(t)&=1+\frac{M_n^{\sss X}(\floor{nt})}{n}-\frac{\floor{nt}}{n}
  -\frac{1}{n}\sum_{i=1}^{\floor{nt}}\delta\bigg(\frac{X_n(i-1)}{n},\frac{Y_n(i-1)}{n}\bigg)+\OP(n^{-1})\\
  &=1+\frac{M_n^{\sss X}(\floor{nt})}{n}-t-\int_0^t\delta(x_n(s),y_n(s))\dif s
  +\OP(n^{-1}).
 \end{split}
\end{equation}
Similar arguments yield
\begin{eq}\label{eq:repr-yn}
 y_n(t)=1+\frac{M_n^{\sss Y}(\floor{nt})}{n}-\mu t +\OP(n^{-1}),
\end{eq}
where $\bld{M}_n^{\sss Y}  =\{M_n^{\sss Y}(t)\}_{t\geq 1}$ is a square-integrable martingale with respect to a suitable filtration. We write $\bld{x}_n$ and $\bld{y}_n$ to denote the processes $(x_n(t))_{t\geq 0}$ an $(y_n(t))_{t\geq 0}$ respectively.

\subsection{Quadratic variation and covariation}\label{sec:qvqcv}

To investigate the scaling behavior of the martingales, we will now compute the respective quadratic variation and covariation terms.
For convenience in notation, denote by $\PP_t, \E_t$, $\mathrm{Var}_t$, $\mathrm{Cov}_t$, the conditional probability, expectation, variance and covariance, respectively,  conditioned on $(X_n(t), Y_n(t))$.
Notice that, for the martingales $\bld{M}_n^{\sss X}$ and $\bld{M}_n^{\sss Y}$, the quadratic variation and covariation terms are given by
\begin{eq}\label{defn:var-covar-mart}
 \langle M_n^{\sss X} \rangle (\floor{nt})&= \sum_{i=1}^{\floor{nt}}\mathrm{Var}_{i-1}(\xi_n(i)),\\
 \quad  \langle M_n^{\sss Y} \rangle (\floor{nt}) &= \sum_{i=1}^{\floor{nt}}\mathrm{Var}_{i-1}(\zeta_n(i)),\\
  \langle M_n^{\sss X}, M_n^{\sss Y} \rangle (\floor{nt})& = \sum_{i=1}^{\floor{nt}}\mathrm{Cov}_{i-1}(\zeta_n(i),\xi_n(i)).
\end{eq}
  Thus,  the quantities of interest are $\mathrm{Var}_t(\xi_n(t+1))$, $\mathrm{Var}_t(\zeta_n(t+1))$ and $\mathrm{Cov}_t(\xi_n(t+1),\zeta_n(t+1))$, which we derive in the three successive claims.

\begin{claim}
For any $t\geq 1$,
  \begin{equation}\label{var-zeta}
   \mathrm{Var}_t(\zeta_n(t+1)) = \sigma^2.
  \end{equation}
\end{claim}
 \begin{proof}
 The proof is immediate by observing that the random variable denoting the household size has variance~$\sigma^2$.
 \end{proof}

\begin{claim}
For any $t\geq 1$,
\begin{eq}\label{var-xi}
 \mathrm{Var}_t\left(\xi_n(t+1)\right) = \frac{X_n(t)(\mu-1)}{Y_n(t)}+\frac{\lambda X_n(t)}{n}+\OP(n^{-1}).
\end{eq}
\end{claim}
\begin{proof}
From the definition of $\xi_n$ in \eqref{defn:xi}, the computation of $\mathrm{Var}_t(\xi_n(t+1))$ requires computation of $\mathrm{Var}_t(\eta_1(t+1))$, $\mathrm{Cov}_t(\eta_1(t+1),\eta_2(t+1))$ and $\mathrm{Var}_t(\eta_2(t+1))$. Since $\eta_1$ follows a Hypergeometric distribution, 
\begin{equation}
\begin{split}
&\E_t\left(\eta_1(t+1)(\eta_1(t+1)-1)\vert H\right)
=\frac{X_n(t)(X_n(t)-1)(H-1)(H-2)}{Y_n(t)(Y_n(t)-1)}
\end{split}
\end{equation}
and
\begin{eq}\label{var-eta1}
 \mathrm{Var}_t\left( \eta_1(t+1) \right)&=\frac{X_n(t)(X_n(t)-1)\expt{(H-1)(H-2)}}{Y_n(t)(Y_n(t)-1)}+\E_t\left(\eta_1(t+1)\right)-\E_t^2\left(\eta_1(t+1)\right)\\
 & = \frac{X_n^2(t)}{Y_n^2(t)}(\sigma^2+\mu^2-3\mu+2)+\frac{X_n(t)}{Y_n(t)}(\mu-1) - \frac{X_n^2(t)}{Y_n^2(t)}(\mu-1)^2+\OP(n^{-1})\\
 & =  \frac{X_n^2(t)}{Y_n^2(t)}(\sigma^2-\mu+1)+\frac{X_n(t)}{Y_n(t)}(\mu-1)+\OP(n^{-1})\\
 & = \frac{X_n(t)}{Y_n(t)}(\mu-1)+\OP(n^{-1}),
\end{eq}
since $\sigma^2=\mu-1=\alpha c.$
Also, we have
\begin{eq}
 &\E_t\left(\eta_2(t+1)(\eta_2(t+1)-1)\right)\\
 &=\big[ (X_n(t)-1)(X_n(t)-2)-\E_t\left(\eta_1(t+1)\right) \left[ 2X_n(t)-3\right]
+\E_t\left(\eta_1^2(t+1)\right)\big]\bigg(\frac{\lambda}{n}\bigg)^2\\
 & = \frac{\lambda^2 X_n^2(t)}{n^2} + \OP(n^{-1})
\end{eq}
and therefore
\begin{eq}\label{var-eta2}
\mathrm{Var}_t(\eta_2(t+1)) = \frac{\lambda X_n(t)}{n} +\OP(n^{-1}).
\end{eq}
Further,
\begin{eq}
 \E_t\left(\eta_1(t+1)\eta_2(t+1)\right)
 &=\E_t\left(\eta_1(t+1)(X_n(t)-1-\eta_1(t+1))\frac{\lambda}{n}\right)\\
 &=\frac{\lambda}{n}\left[ (X_n(t)-1)\E_t(\eta_1(t+1))-\E_t(\eta_1^2(t+1))\right]\\
  & = \frac{\lambda X_n(t)}{n}\frac{X_n(t)(\mu-1)}{Y_n(t)} +\OP(n^{-1}).
\end{eq}
Now, from \eqref{defn:eta1}, \eqref{defn:eta2}, 
\begin{equation*}
\E_t(\eta_1(t+1)) = \frac{\lambda X_n(t)}{n}+\OP(n^{-1}) \quad \text{and}\quad  \E_t(\eta_2(t+1)) = \frac{X_n(t)(\mu-1)}{Y_n(t)}+\OP(n^{-1}),
\end{equation*}which implies that 
\begin{eq}\label{covar-eta12}
\mathrm{Cov}_t(\eta_1(t+1),\eta_2(t+1)) = \OP(n^{-1}).
\end{eq}
Combining \eqref{var-eta1}, \eqref{var-eta2} and \eqref{covar-eta12}, gives~\eqref{var-xi}.
\end{proof}

\begin{claim}
For any $t\geq 1$,
\begin{eq}\label{covar-zeta-xi}
\mathrm{Cov}_t\left( \zeta_n(t+1),\xi_n(t+1)\right) = \frac{X_n(t)}{Y_n(t)}\sigma^2+\OP(n^{-1}).
\end{eq}
\end{claim}
\begin{proof}
Observe that 
\begin{eq}
 &\E_t\left(\zeta_n(t+1)\eta_1(t+1)\right)=\E_t\left(\zeta_n(t+1) \E_t\left(\eta_1(t+1)|\zeta_n(t+1)\right)\right)\\
& =\frac{X_n(t)}{Y_n(t)}\E_t\left(\zeta_n(t+1)(\zeta_n(t+1)-1)\right)=\frac{X_n(t)}{Y_n(t)}(\sigma^2+\mu^2-\mu),
\end{eq}
and therefore,
\begin{equation}\label{covar-zeta-eta1}
 \mathrm{Cov}_t(\zeta_n(t+1),\eta_1(t+1))= \frac{X_n(t)}{Y_n(t)}\sigma^2.
\end{equation}
Thus,
\begin{eq}
 &\exptt{\zeta_n(t+1)\eta_2(t+1)} = \exptt{\zeta_n(t+1)\exptt{\eta_2(t+1)\vert \eta_1(t+1),\zeta_n(t+1)}}\\
 & = \exptt{\zeta_n(t+1)(X_n(t)-1-\eta_1(t+1))\frac{\lambda}{n}} =  \lambda\mu \frac{X_n(t)}{n}+\OP(n^{-1})
\end{eq}
and hence
\begin{eq}\label{covar-zeta-eta2}
 \mathrm{Cov}_t\left( \zeta_n(t+1),\eta_2(t+1)\right) = \OP(n^{-1}).
\end{eq}
Combining \eqref{covar-zeta-eta1} and \eqref{covar-zeta-eta2} yields~\eqref{covar-zeta-xi}.
\end{proof}

Based on the quadratic variation and covariation results above, the following lemma shows that the martingales when scaled by $n$, converge to the zero-process.
\begin{lemma}\label{lem:qv-order}
For any fixed $T\geq 0$, as $n\to\infty$,
\begin{eq}
\frac{1}{n}\sup_{t\leq T}|\bld{M}_n^{\sss X}(\floor{nt})| \pto 0, \quad \frac{1}{n}\sup_{t\leq T}|\bld{M}_n^{\sss Y}(\floor{nt})| \pto 0.
\end{eq}
\end{lemma}
\begin{proof}
Observe that using \eqref{defn:var-covar-mart} along with \eqref{var-zeta} and \eqref{var-xi}, we can claim for any $T\geq 0$,
\begin{eq}
\langle \bld{M}_n^{\sss X}  \rangle (\floor{nT})=\OP(n),\quad
\langle \bld{M}_n^{\sss Y}  \rangle (\floor{nT})=\OP(n).
\end{eq}
Thus, from Doob's inequality \cite[Theorem 1.9.1.3]{LS89}, the proof follows.
\end{proof}
\subsection{Convergence of the scaled exploration process}
\label{sec:fluidconv}

Based on the estimates from Sections~~\ref{sec:mart-decompose},~and~\ref{sec:qvqcv}, we now complete the proof of Theorem~\ref{th:fluid}.
 Recall the representations of $\bld{x}_n$, and $\bld{y}_n$ from \eqref{eq:mart-decompose} and \eqref{eq:repr-yn}.
 Fix any $ 0 \leq T < 1/\mu$.
 Observe that Lemma~\ref{lem:qv-order} immediately yields
 \begin{eq}\label{fluid-y}
 \sup_{t\leq T} |y_n(t)-y(t)| \pto 0.
 \end{eq}
  Next note that $\delta (x,y)$, as defined in~\eqref{eq:delta}, is Lipschitz continuous on $[0,1]\times[\epsilon, 1]$ for any $\epsilon >0$ and
 we can choose this $\epsilon > 0$ in such a way that $y(t)\geq \epsilon$ for all $t\leq T$ (since $T<1/\mu$).
 Therefore, the Lipschitz continuity of $\delta$ implies that  there exists a constant $C>0$ such that
 \begin{align*}
 &\sup_{t\leq T}|\delta(x_n(t),y_n(t))-\delta(x(t),y(t))|
\leq C \Big(\sup_{t\leq T}| x_n(t)-x(t)|+\sup_{t\leq T}| y_n(t)-y(t)|\Big).
 \end{align*}Thus,
 \begin{align*}
  \sup_{t\leq T} |x_n(t)-x(t)|  &\leq \sup_{t\leq T}\frac{|M_n^{\sss X}(\floor{nt})|}{n} 
  +\int_0^T \sup_{t\leq u}  |\delta(x_n(t),y_n(t))
  -\delta(x(t),y(t))| \dif u +\oP(1)\\
  &\leq \varepsilon_n + C \int_0^T \sup_{t\leq u}  | x_n(t)-x(t)| \dif u,
 \end{align*}
 where, by Lemma~\ref{lem:qv-order} and \eqref{fluid-y},
 \begin{align*}
  \varepsilon_n :=  \sup_{t\leq T}\frac{|M_n^{\sss X}(\floor{nt})|}{n} + CT \sup_{t\leq T} |y_n(t)-y(t)|+\oP(1),
 \end{align*}
which converges in probability to zero, as $n\to\infty$.
Using Gr\H{o}nwall's inequality~\cite[Theorem~5.1]{EK2009}, we get
\begin{equation}\label{eq:fluid-conv-eqn}
 \sup_{t\leq T}|x_n(t)-x(t)|\leq  \varepsilon_n \e^{CT}\pto 0.
\end{equation}
Finally, due to Claim~\ref{cl:upperbound} below we note that the smallest root of $x(t)$ is strictly smaller than $1/\mu$.
Also, the convergence in \eqref{eq:fluid-conv-eqn} holds for any $T<1/\mu$.
This concludes the proof of Theorem~\ref{th:fluid}.\qed


The claim below establishes that $J^{\star}<1/\mu$.
\begin{claim}\label{cl:upperbound}
$J^\star<1/\mu$.
\end{claim}
\begin{proof}
Recall that $\mu = (1+\alpha c)$ and $\lambda = (1-\alpha) c$.
 Notice that \eqref{diffeq} gives a linear differential equation, and the solution is given by
\begin{eq}\label{eq:ode-sol}
x(t) = \e^{-\lambda t} (1-\mu t)^{\frac{\mu-1}{\mu}}\bigg(1-\int_0^t\e^{\lambda s}(1-\mu s)^{-1+\frac{1}{\mu}}\dif s\bigg), \quad t<\frac{1}{\mu}.
\end{eq} 
Thus, the smallest root of the integral equation \eqref{diffeq} defined as $J^\star$  must be the smallest positive solution of 
\begin{equation}\label{eq:sol-dif-eqn}
 \mathcal{I}(t)=\int_0^t\e^{\lambda s}(1-\mu s)^{-1+\frac{1}{\mu}}\dif s = 1.
\end{equation}
The integrand in the left hand side of \eqref{eq:sol-dif-eqn} is positive, and tends to $\infty$ as $t$ increases to $1/\mu$. 
Therefore, the integral $ \int_0^t\e^{\lambda s}(1-\mu s)^{-1+1/\mu}\dif s$ tends to infinity as well.
Thus, there must exist a solution of \eqref{eq:sol-dif-eqn} which is smaller that $1/\mu$.
This in turn implies that $J^\star<\mu^{-1}$.
\end{proof}
We now complete the proof of Corollary~\ref{cor:J-large-c}.
\begin{proof}[Proof of Corollary~\ref{cor:J-large-c}] 
Observe from~\eqref{eq:ode-sol} and \eqref{eq:sol-dif-eqn} that, for $t<\mu^{-1}$, 
 \begin{align*}
  &\mathcal{I}(t)\geq  
  1 - \e^{\lambda t}(1-\mu t)^{\frac{1}{\mu}}\int_0^t \frac{\dif s}{1-\mu s}
  \geq 1-\e^{\frac{\lambda}{\mu}} \bigg[-\frac{1}{\mu} \log(1-\mu s)\bigg]_0^t \sim 1-\e^{\frac{\lambda}{\mu}}  \frac{1}{\mu}\log(1-\mu t),
 \end{align*}
 and the last term is zero when
 \begin{eq}
  t = \frac{1}{\mu} (1-\e^{-\mu \e^{-\frac{\lambda}{\mu}}}) \sim \frac{1}{\mu} \text{ as } \mu\to\infty. 
 \end{eq}Since $J^\star \geq (1-\e^{-\mu \e^{-\frac{\lambda}{\mu}}})$, the proof is complete.
\end{proof}

\section{Proof of Theorem~\texorpdfstring{\ref{th:jam-diffusion}}{3.2}}\label{sec:diffconv}
Define the diffusion-scaled processes
\begin{eq}\label{eq:diff-scale}
\bar{X}_n(t):=\sqrt{n}(x_n(t)-x(t)), \quad
\bar{Y}_n(t):=\sqrt{n}(y_n(t)-y(t)),
\end{eq}
and the diffusion-scaled martingales
$$\bar{M}_n^{\sss X} (t):= \frac{M_n^{\sss X}(\floor{nt})}{\sqrt{n}},\quad\bar{M}_n^{\sss Y} (t):= \frac{M_n^{\sss Y}(\floor{nt})}{\sqrt{n}}.$$
Now observe from~\eqref{eq:mart-decompose} that
\begin{align*}
\bar{X}_n(t)
&=\bar{M}_n^{\sss X}(t)-(\mu-1)\left[ \int_0^t\frac{\bar{X}_n(s)}{y_n(s)}\dif s 
+\int_0^tx(s)\sqrt{n}\left(\frac{1}{y_n(s)}-\frac{1}{y(s)}\right)\dif s\right]\\
&\hspace{7cm} -\lambda\int_0^t\bar{X}_n(s)\dif s+\OP(n^{-1/2})\\
&=\bar{M}_n^{ \sss X}(t)-(\mu-1)\int_0^t\frac{\bar{X}_n(s)}{y_n(s)}\dif s 
+  \int_o^t\frac{x(s)(\mu-1)}{y_n(s)y(s)}\bar{Y}_n(s) \dif s\\
&\hspace{7cm}-\lambda\int_0^t\bar{X}_n(s)\dif s+\oP(1).
\end{align*}
Therefore, we can write
\begin{eq}\label{diff-xn-simpl}
\bar{X}_n(t) &= \bar{M}_n^{\sss X}(t) + \int_0^t f_n(s)\bar{X}_n(s)\dif s 
+ \int_0^t g_n(s)\bar{Y}_n(s)\dif s + \oP(1),
\end{eq}
where
\begin{eq}
 f_n(t) = -\frac{(\mu-1)}{y_n(t)} -\lambda, \quad g_n(t) = \frac{(\mu-1)x(t)}{y_n(t)y(t)}.
\end{eq}
Furthermore,~\eqref{eq:repr-yn} yields
\begin{eq}\label{diff-yn-simpl}
 &\bar{Y}_n(t) =  \sqrt{n}(y_n(t)-y(t)) =  \bar{M}_n^{\sss Y}(t) +\oP(1).
\end{eq}
Based on the quadratic variation and covariation results in Section~\ref{sec:qvqcv}, the following lemma shows that the martingales when scaled by $\sqrt{n}$ converge to a diffusion process described by an SDE.
\begin{lemma}[{Diffusion limit of martingales}]
\label{lem:conv-mart}
As $n\to\infty$, $(\bar{\bld{M}}_n^{\sss X}, \bar{\bld{M}}_n^{\sss Y})\dto (\bld{W}_1,\bld{W}_2)$, where the process $(\bld{W}_1,\bld{W}_2)$ is described by the {\rm SDE}
\begin{eq}
 \dif W_1(t) = \sqrt{\beta(t)}\left[ \rho(t)\dif B_1(t)+\sqrt{1-\rho(t)^2}\dif B_2(t) \right], \quad \dif W_2(t) = \sigma \dif B_1(t)
\end{eq}with $\bld{B}_1$ and $\bld{B}_2$ two independent standard Brownian motions.
\end{lemma}
\begin{proof}
 The idea is to use the martingale functional central limit theorem (cf.~\cite[Theorem 8.1]{PTRW07}), where the convergence of the martingales is characterized by the convergence of their quadratic variation process.
 Using Theorem~\ref{th:fluid}, we compute the asymptotics of the quadratic variations and covariation of $\bar{\bld{M}}_n^{\sss X}$ and $\bar{\bld{M}}_n^{\sss Y}$.
From \eqref{eq:mart-decompose} and \eqref{var-zeta}, we obtain
\begin{align*}
 \langle \bar{M}_n^{\sss Y} \rangle (t) = \frac{1}{n} \sum_{i=1}^{\floor{nt}}\mathrm{Var}_{i-1}(\zeta_n(i)) \pto \sigma^2 t.
\end{align*}
Again, \eqref{eq:mart-decompose}, \eqref{var-xi} and Theorem~\ref{th:fluid} yields
\begin{align*}
 \langle \bar{M}_n^{\sss X} \rangle (t)
=\frac{1}{n}\sum_{i=1}^{\floor{nt}}\mathrm{Var}_{i-1}(\xi_n(i))
  \pto
 \int_0^t\left[\frac{(\mu-1)}{y(s)}+\lambda\right] x(s)\dif s 
 = \int _0^t\beta(s)\dif s.
\end{align*}
Finally, from \eqref{eq:mart-decompose}, \eqref{covar-zeta-xi} and Theorem~\ref{th:fluid} we obtain
\begin{align*}
\langle \bar{M}_n^{\sss X}, \bar{M}_n^{\sss Y} \rangle (t)=
\frac{1}{n} \sum_{i=1}^{\floor{nt}}\mathrm{Cov}_{i-1}(\zeta_n(i),\xi_n(i))
 \pto \sigma^2\int_0^t\frac{x(s)}{y(s)}\dif s
   = \int_0^t \rho(s)\times \sigma \sqrt{\beta(s)}\dif s.
\end{align*}
From the martingale functional central limit theorem, 
we get that $(\bar{\bld{M}}_n^{\sss X}, \bar{\bld{M}}_n^{\sss Y})\dto (\hat{\bld{W}}_1,\hat{\bld{W}}_2)$, where $(\hat{\bld{W}}_1,\hat{\bld{W}}_2)$ are Brownian motions with zero means and quadratic covariation matrix
$$\begin{bmatrix}
\int _0^t\beta(s)\dif s & \int_0^t \rho(s)\times \sigma \sqrt{\beta(s)}\dif s\\
\int_0^t \rho(s)\times \sigma \sqrt{\beta(s)}\dif s & \sigma^2 t
\end{bmatrix}.
$$
The proof then follows by noting the fact that $(\bld{W}_1,\bld{W}_2)\disteq (\hat{\bld{W}}_1,\hat{\bld{W}}_2).$
\end{proof}
Having proved the above convergence of martingales, we now establish weak convergence of the scaled exploration process to a suitable diffusion process.
\begin{proposition}[{Functional CLT of the exploration process}]
\label{prop-diffusion}
 As $n\to\infty$, $(\bar{\bld{X}}_n,\bar{\bld{Y}}_n)\dto (\bld{X},\bld{Y})$ where $(\bld{X},\bld{Y})$ is the two-dimensional stochastic process satisfying the SDE
 \begin{eq}\label{eq:diffusion-SDE}
  \dif X(t)&=\sqrt{\beta(t)}\left[ \rho(t)\dif B_1(t)+\sqrt{1-\rho(t)^2}\dif B_2(t) \right]
  +f(t)X(t)\dif t + g(t)Y(t) \dif t,\\
  \dif Y(t)&=\sigma\dif B_1(t),
 \end{eq}
 with $\bld{B}_1$, $\bld{B}_2$ being independent standard Brownian motions, and $f(t)$, $g(t)$ and $\rho(t)$  as defined in~\eqref{defn:functions}.
\end{proposition}
\begin{proof}
First we show that  $((\bar{\bld{X}}_n,\bar{\bld{Y}}_n))_{n\geq 1}$ is a stochastically bounded sequence of processes.
Indeed stochastic boundedness (and in fact weak convergence) of the $\bar{\YY}_n$ process follows from Lemma~\ref{lem:conv-mart}.
Further observe that for any $T<1/\mu$, by Theorem~\ref{th:fluid},
\begin{eq}\label{eq:limit-fg}
 \sup_{t\leq T} |f_n(t)-f(t)| \pto 0, \qquad \sup_{t\leq T} |g_n(t)-g(t)| \pto 0,
\end{eq}
where $f,g$ are defined in \eqref{defn:functions}.
Therefore, for any $T<1/\mu$,
 \begin{align*}
  \sup_{t\leq T} |\bar{X}_n(t)|&\leq \sup_{t\leq T} |\bar{M}_n^{\sss X}(t)| +T\sup_{t\leq T}|g_n(t)\bar{Y}_n(t)|
+ \sup_{t\leq T}|f_n(t)| \int_0^T \sup_{u\leq t} |\bar{X}_n(u)| \dif t,
 \end{align*}
and again using Gr\H{o}nwall's inequality, it follows that 
\begin{align*}
 \sup_{t\leq T} |\bar{X}_n(t)| &\leq \Big(\sup_{t\leq T} |\bar{M}_n^{\sss X}(t)| +T\sup_{t\leq T}|g_n(t)\bar{Y}_n(t)|\Big)
\times \exp\Big(T\sup_{t\leq T}|f_n(t)|\Big).
\end{align*}
Then stochastic boundedness of $(\bar{\bld{X}}_n)_{n\geq 1}$ follows from Lemma~\ref{lem:conv-mart}, \eqref{eq:limit-fg},
and the stochastic boundedness criterion for square-integrable martingales given in~\cite[Lemma 5.8]{PTRW07}.

From stochastic boundedness of the processes we can claim that any sequence $(n_k)_{k\geq 1}$ has a further subsequence $(n_k')_{k\geq 1}\subseteq (n_k)_{k\geq 1}$ such that
 \begin{eq}
 (\bar{\bld{X}}_{n'_k},\bar{\bld{Y}}_{n'_k})\dto (\bld{X}',\bld{Y}'),
 \end{eq}
along that subsequence, where the limit  $(\bld{X}',\bld{Y}')$ may depend on the subsequence $(n_k)_{k\geq 1}$.
 However, due to the convergence result in Lemma~\ref{lem:conv-mart} and \eqref{eq:limit-fg}, the continuous mapping theorem (see~\cite[Section~3.4]{W02}) implies that the limit $(\bld{X}',\bld{Y}')$ must satisfy \eqref{eq:diffusion-SDE}.
 Again, the solution to the SDE in~\eqref{eq:diffusion-SDE} is unique, and
 therefore the limit $(\bld{X}',\bld{Y}')$ does not depend on the subsequence $(n_k)_{k\geq 1}$. Thus, the proof is complete.
 \end{proof}

 \begin{proof}[Proof of Theorem~\ref{th:jam-diffusion}]
 First observe that
 $$\sqrt{n}(J_n^\star-J^\star)\dto X(J^\star)\quad\text{as}\quad n\to\infty.$$
 Indeed this can be seen by the application of the hitting time distribution theorem in~\cite[Theorem~4.1]{EK2009}, and noting the fact that $x'(J^\star)=-1$.
Now since $\XX$ is a centered Gaussian process, in order to complete the proof of Theorem~\ref{th:jam-diffusion}, we only need to compute $\mathrm{Var}(X(J^\star))$.
We will use the following known result~\cite[Theorem~8.5.5]{A74} to calculate the variance of $X(t)$.
\begin{lemma}[{Expectation and variance of SDE}]
Consider the $d$-dimensional stochastic differential equation given by
\begin{equation}
 \dif Z(t) = (A(t)Z(t)+a(t))\dif t+ \sum_{i=1}^db_i(t)\dif B_i(t),
\end{equation}
where $Z(0)=z_0\in \R^d$, the $b_i$'s are $\R^d$-valued functions, and the $B_i$'s are independent standard Brownian motions, $i=1,\ldots, d$. Then given $Z(0)=x_0$, $Z(t)$ has a normal distribution with mean vector $m(t)$ and covariance matrix $V(t)$, where $m(t)$ and $V(t)$ satisfy the recursion relations
\begin{eq}
 \frac{\dif}{\dif t} m (t) = A(t)m(t) + a(t),\quad
 \frac{\dif}{\dif t} V(t) = A(t)V (t) + V (t)A^{\sss T} (t) +\sum_{i=1}^d b_ib_i(t)^{\sss T},  
\end{eq}with initial conditions  $m(0) = x_0$, and   $V (0) = 0$.
\end{lemma}
In our case, observe from~\eqref{eq:diffusion-SDE} that
\begin{eq}
A(t)=\begin{bmatrix}
 f(t) & g(t)\\
 0 & 0
 \end{bmatrix}, \
 a(t)=\begin{bmatrix}
  0\\0
 \end{bmatrix},\
 b_1(t)=
 \begin{bmatrix}
  \rho(t)\sqrt{\beta(t)}\\
  \sigma
 \end{bmatrix},\ 
 b_2(t)=
 \begin{bmatrix}
  \sqrt{1-\rho(t)^2}\sqrt{\beta(t)}\\
  0
 \end{bmatrix}
\end{eq}
Denote the variance-covariance matrix of $(X(t),Y(t))$ by
\begin{equation}
 V(t) = \begin{bmatrix}
  \sigma_{xx}(t) & \sigma_{xy}(t)\\
  \sigma_{xy}(t)& \sigma_{yy}(t)
 \end{bmatrix}.
\end{equation}
Then
\begin{align*}
 \frac{\dif }{\dif t} V(t)
 &=
 \begin{bmatrix}
  \sigma_{xx}(t)f(t)+\sigma_{xy}(t)g(t)& \sigma_{xy}(t)f(t)+\sigma_{yy}(t)g(t)\\
  0 & 0
 \end{bmatrix}
 +
 \begin{bmatrix}
  \sigma_{xx}(t)f(t)+\sigma_{xy}(t)g(t)& 0\\
  \sigma_{xy}(t)f(t)+\sigma_{yy}(t)g(t) & 0
 \end{bmatrix}\\
 &\hspace{1cm}+
 \begin{bmatrix}
  \rho(t)^2\beta(t)&\sqrt{\beta(t)}\sigma\rho(t)\\
  \sqrt{\beta(t)}\sigma\rho(t)&\sigma^2
 \end{bmatrix}
 +\begin{bmatrix}
  (1-\rho(t)^2)\beta(t)&0\\
  0&0
 \end{bmatrix}\\
 &=\begin{bmatrix}
  2\sigma_{xx}(t)f(t)+2\sigma_{xy}(t)g(t)&\sigma_{xy}(t)f(t)+\sigma_{yy}(t)g(t)\\
  \sigma_{xy}(t)f(t)+\sigma_{yy}(t)g(t)& 0
 \end{bmatrix}
+
 \begin{bmatrix}
  \beta(t)&\sqrt{\beta(t)}\sigma\rho(t)\\
  \sqrt{\beta(t)}\sigma\rho(t)&\sigma^2
 \end{bmatrix}
\end{align*}
Therefore, the variance of $X(t)$ can be obtained from the solution of the recursion equations
\begin{eq}
  \frac{\dif \sigma_{xx}(t)}{\dif t}&=2(\sigma_{xx}(t)f(t)+\sigma_{xy}(t)g(t))+\beta(t),\\ 
  \frac{\dif \sigma_{xy}(t)}{\dif t}&=\sigma_{xy}(t)f(t)+\sigma_{yy}(t)g(t)+\sqrt{\beta(t)}\sigma\rho(t),
\end{eq}
and the proof is thus completed by noting that $\sigma_{yy}(t)=\sigma^2t$.
 \end{proof}

\section{Clustering coefficient of random geometric graphs}
\label{sec:clustering-coeff}
The clustering coefficient for the random geometric graph was derived in \cite{DC02} along with an asymptotic formula, when the dimension becomes large. 
Below we give an alternative derivation. 
The formula~\eqref{eq:alpha-choice} is more tractable in all dimensions compared with the formula in~\cite{DC02}.
Consider $n$ uniformly chosen points on a $d$-dimensional box $[0,1]^d$ and connect two points $u,v$ by an edge if they are at most $2r$ distance apart. 
Fix any three vertex indices $u,v$, and $w$. 
We write $u\leftrightarrow v$ to denote that $u$ and $v$ share an edge.  
The clustering coefficient for $\RGG(c,d)$ on $n$ vertices is then defined by
\begin{equation}
 C_n(c,d) := \prob{v \leftrightarrow w|u\leftrightarrow v, u\leftrightarrow w}.
\end{equation}
 The following proposition explicitly characterizes the asymptotic value of $C_n(c,d)$ for any density $c$ and dimension $d$.
\begin{proposition}\label{th:cc}
For any fixed $c> 0$, and $d\geq 1$, as $n\to\infty,$
\begin{equation}\label{eq:cc}
C_n(c,d)\to C(d)=d\int_{0}^{1}x^{d-1}I_{1-\frac{x^2}{4}}\Big(\frac{d+1}{2}, \frac{1}{2}\Big) \dif x.
\end{equation}
\end{proposition}
\begin{proof}
Observe that the $\RGG$ model can be constructed by throwing points sequentially at uniformly chosen locations independently, and then connecting to the previous vertices that are at most $2r$ distance away. 
Since the locations of the vertices are chosen independently, without loss of generality we assume that in the construction of the $\RGG$ model, the locations of $u,v,w$ are chosen in this respective order. 
Now, the event $\{u\leftrightarrow v,u\leftrightarrow w,v \leftrightarrow w\}$ occurs if and only if  $v$ falls within the $2r$ neighborhood of $u$ and $w$ falls within the intersection region of two spheres of radius $2r$, centered at $u$ and $v$, respectively. 
Let $B_d(\bld{x},2r)$ denote the $d$-dimensional sphere with radius $2r$, centered at $\bld{x}$, and let $V_d(2r)$ denote its volume. 
Since $r$ is sufficiently small, so that $B_d(\bld{x},2r)\subseteq [0,1]^d$,  using translation invariance, we assume that the location of $u$ is $\bld{0}$. 
Let $\bld{v},\bld{w}$ denote the positions in the $d$-dimensional space, of vertices $v$ and $w$, respectively.  
Notice that, conditional on the event $\{\bld{v}\in B_d(\bld{0},2r)\}$, the position $\bld{v}$ is uniformly distributed over  $B_d(\bld{0},2r)$.
Let $\bld{V}$ be a point chosen uniformly from $B_d(\bld{0},2r)$. 
Then the above discussion yields 
\begin{align}\label{eq:prob-area}
 &C_n(c,d) =\frac{\prob{u\leftrightarrow v,u\leftrightarrow w,v \leftrightarrow w}}{\prob{u\leftrightarrow v, u\leftrightarrow w}}\nonumber\\
 &= \frac{1}{(V_d(2r))^2}\int_{\bld{v}\in B_d(\bld{0},2r)}\prob{\bld{w}\in B_d(\bld{0},2r) \cap B_d(\bld{v},2r)}\dif \bld{v}\nonumber\\
 & = \frac{1}{V_d(2r)} \E[|B_d(\bld{0},2r)\cap B_d(\bld{V},2r)|].
\end{align}
We shall use the following lemma to compute the expectation term in \eqref{eq:prob-area}.
\begin{lemma}[{\cite{sphere}}] For any $\bld{x}$ with $\|\bld{x}\| = \rho$, the intersection volume $|B_d(\bld{0},2r)\cap B_d(\bld{x},2r)|$ depends only on $\rho$ and $r$, and is given by 
\begin{equation}
 |B_d(\bld{0},2r)\cap B_d(\bld{x},2r)| = V_d(2r)\cdot I_{1-\frac{\rho^2}{16r^2}}\Big(\frac{d+1}{2}, \frac{1}{2}\Big),
\end{equation} where $I_z(a,b)$  denotes the normalized incomplete beta integral given by $$ I_z(a,b) = \frac{\int_0^z y^{a-1}(1-y)^{b-1}\dif y}{\int_0^1y^{a-1}(1-y)^{b-1} \dif y}.$$
\end{lemma}

Observe that the Jacobian corresponding to the transformation from the Cartesian coordinates $(x_1,\ldots,x_d)$ to the Polar coordinates $(\rho,\theta,\phi_1,\ldots,\phi_{d-2})$, is given by 
\begin{equation}
 J_d (\rho, \theta,\phi_1,\dots,\phi_{d-2}) = \rho^{d-1} \prod_{j=1}^{d-2} (\sin(\phi_j))^{d-1-j}.  
\end{equation}
Thus, \eqref{eq:prob-area} reduces to 
\begin{align*}
&C_n(c,d) = \frac{1}{(V_d(2r))^2} \int_{\bld{x}\in B_d(\bld{0},2r)}|B_d(\bld{0},2r)\cap B_d(\bld{x},2r)| \dif \bld{x}= \frac{1}{V_d(2r)} \int_{\|\bld{x}\|\leq 2r}I_{1-\frac{\|\bld{x}\|^2}{16r^2}}\Big(\frac{d+1}{2}, \frac{1}{2}\Big) \dif \bld{x}\\
  &= \frac{1}{V_d(2r)}\int_{0}^{2r}\int_{0}^{2\pi} \int_{0}^\pi\dots \int_{0}^{\pi} \rho^{d-1}I_{1-\frac{\rho^2}{16r^2}}\Big(\frac{d+1}{2}, \frac{1}{2}\Big)\prod_{j=1}^{d-2} (\sin(\phi_j))^{d-1-j}\prod_{j=1}^{d-2}\dif \phi_j \dif \theta \dif \rho,
\end{align*}
and we obtain,
\begin{align*}
C_n(c,d) = \Big(\int_{0}^{2r}\rho^{d-1}\dif \rho\Big)^{-1}
\int_{0}^{2r}\rho^{d-1}I_{1-\frac{\rho^2}{16r^2}}\Big(\frac{d+1}{2}, \frac{1}{2}\Big)\dif \rho,
\end{align*}
since
$$V_d(2r) = 2\pi \bigg(\prod_{j=1}^{d-2}\int_0^\pi (\sin(\phi_j))^{d-1-j} \dif \phi_j\bigg) \int_0^{2r}\rho^{d-1}\dif \rho.$$
Therefore, putting $x = \rho/2r$, yields
\begin{align}
 C_n(c,d)  &= \Big(\int_0^1 x^{d-1}\dif x\Big)^{-1}\int_{0}^{1}x^{d-1}I_{1-\frac{x^2}{4}}\big(\frac{d+1}{2}, \frac{1}{2}\big)\dif x = d\int_{0}^{1}x^{d-1}I_{1-\frac{x^2}{4}}\big(\frac{d+1}{2}, \frac{1}{2}\big)\dif x,
\end{align}which proves the result.
\end{proof}

\section{Discussion}\label{sec:discussion}
We introduced a clustered random graph model with tunable local clustering and a sparse superimposed structure. 
The level of clustering
was set to suitably match the local clustering in the topology generated by the random geometric graph. 
This resulted in a unique parameter $\alpha_d$ that for each dimension $d$ creates a one-to-one mapping between the tractable random network model and the intractable random geometric graph. 
In this way, we offer a new perspective for understanding $\RSA$ on the continuum space in terms of \RSA on random networks with local clustering. 
Analysis of the random network model resulted in precise characterizations of the limiting jamming fraction and its fluctuation. 
The precise results then served, using the one-to-one mapping, as predictions for the fraction of covered volume for \RSA in the Euclidean space. 
Based on extensive simulations we then showed that these prediction were remarkably accurate, irrespective of density or dimension.



In our analysis the random network model serves as a topology generator that replaces the topology generated by the random geometric graph. 
While the latter is directly connected with the metric in the Euclidian space, the only spatial information in the topologies generated by the random network model is contained in the matched average degree and clustering. 
%
%
One could be inclined to think that random topology generators such as the $\CRG(c,\alpha_d)$ model may not be good enough. Indeed, this random network model reduces all possible interactions among pairs of vertices to only two principal components:  the local interactions due to the clustering, and a mean-field distant interaction. There is, however, building evidence that such randomized topologies can approximate rigid spatial topologies when the local interactions in both topologies are matched. Apart from this paper, the strongest evidence to date for this line of reasoning is \cite{K2016}, where it was shown that the typical ensembles from the latent-space geometric graph model can be modeled by an inhomogeneous random graph model that matches with the original graph in terms of the average degree and a measure of clustering. 
We should mention that \cite{K2016} is restricted to one-dimensional models and does not deal with $\RSA$, but it shares with this paper the perspective that matching degrees and local clustering can be sufficient for describing spatial settings.

\bibliographystyle{abbrv}
\bibliography{references2}

\begin{thebibliography}{10}

\bibitem{A74}
L.~Arnold.
\newblock {\em {Stochastic Differential Equations: Theory and Applications}}.
\newblock John Wiley {\&} Sons, 1974.

\bibitem{BBS13}
F.~Ball, T.~Britton, and D.~Sirl.
\newblock {A network with tunable clustering, degree correlation and degree
  distribution, and an epidemic thereon}.
\newblock {\em J. Math. Biol.}, 66(4):979--1019, 2013.

\bibitem{B11}
M.~Barth{\'{e}}lemy.
\newblock {Spatial networks}.
\newblock {\em Phys. Rep.}, 499(1):1--101, 2011.

\bibitem{BJM13}
P.~Bermolen, M.~Jonckheere, and P.~Moyal.
\newblock {The jamming constant of uniform random graphs}.
\newblock {\em arXiv:1310.8475}, 2013.

\bibitem{BJS15}
P.~Bermolen, M.~Jonckheere, and J.~Sanders.
\newblock {Scaling limits for exploration algorithms}.
\newblock {\em arXiv:1504.02438}, 2015.

\bibitem{BJL15}
G.~Brightwell, S.~Janson, and M.~Luczak.
\newblock {The greedy independent set in a random graph with given degrees}.
\newblock {\em arXiv:1510.05560}, 2015.

\bibitem{CAP07}
A.~Cadilhe, N.~A.~M. Ara{\'{u}}jo, and V.~Privman.
\newblock {Random sequential adsorption: from continuum to lattice and
  pre-patterned substrates}.
\newblock {\em J. Phys. Condens. Matter}, 19(6):65124, 2007.

\bibitem{CKE2002}
J.~Y. Chen, J.~F. Klemic, and M.~Elimelech.
\newblock {Micropatterning microscopic charge heterogeneity on flat surfaces
  for studying the interaction between colloidal particles and heterogeneously
  charged surfaces}.
\newblock {\em Nano Lett.}, 2(4):393--396, 2002.

\bibitem{CL14}
E.~Coupechoux and M.~Lelarge.
\newblock {How clustering affects epidemics in random networks}.
\newblock {\em Adv. Appl. Probab.}, 46(4):985--1008, 2014.

\bibitem{DC02}
J.~Dall and M.~Christensen.
\newblock {Random geometric graphs}.
\newblock {\em Phys. Rev. E}, 66(1):016121, 2002.

\bibitem{DGPLCM2002}
L.~M. Demers, D.~S. Ginger, S.~J. Park, Z.~Li, S.~W. Chung, and C.~A. Mirkin.
\newblock {Direct patterning of modified oligonucleotides on metals and
  insulators by dip-pen nanolithography}.
\newblock {\em Science}, 296(5574):1836--1838, 2002.

\bibitem{DGLU11}
L.~Devroye, A.~Gy{\"{o}}rgy, G.~Lugosi, and F.~Udina.
\newblock {High-dimensional random geometric graphs and their clique number}.
\newblock {\em Electron. J. Probab.}, 16(0):2481--2508, 2011.

\bibitem{DLM16}
S.~Dhara, J.~S.~H. van Leeuwaarden, and D.~Mukherjee.
\newblock {Generalized Random Sequential Adsorption on Erd\H{o}s--R\'enyi
  Random Graphs}.
\newblock {\em J. Stat. Phys.}, 164(5):1217--1232, 2016.

\bibitem{EK2009}
S.~N. Ethier and T.~G. Kurtz.
\newblock {\em {Markov Processes: Characterization and Convergence}}.
\newblock John Wiley {\&} Sons, 2009.

\bibitem{F1980}
J.~Feder.
\newblock {Random sequential adsorption}.
\newblock {\em J. Theor. Biol.}, 87(2):237--254, 1980.

\bibitem{GMS2012}
D.~C. Gijswijt, H.~D. Mittelmann, and A.~Schrijver.
\newblock {Semidefinite code bounds based on quadruple distances}.
\newblock {\em IEEE Trans. Inf. Theory}, 58(5):2697--2705, 2012.

\bibitem{HLS15}
R.~van~der Hofstad, J.~S.~H. van Leeuwaarden, and C.~Stegehuis.
\newblock {Hierarchical Configuration Model}.
\newblock {\em arXiv:1512.08397}, 2015.

\bibitem{JCZRCL2000}
D.~Jaksch, J.~I. Cirac, P.~Zoller, S.~L. Rolston, R.~C{\^{o}}t{\'{e}}, and
  M.~D. Lukin.
\newblock {Fast Quantum Gates for Neutral Atoms}.
\newblock {\em Phys. Rev. Lett.}, 85(10):2208--2211, sep 2000.

\bibitem{KS91}
I.~Karatzas and S.~E. Shreve.
\newblock {\em {Brownian Motion and Stochastic Calculus}}, volume 113.
\newblock Springer-Verlag, New York, 1991.

\bibitem{KN10}
B.~Karrer and M.~E.~J. Newman.
\newblock {Random graphs containing arbitrary distributions of subgraphs}.
\newblock {\em Phys. Rev. E}, 82(6):066118, 2010.

\bibitem{KT2013}
H.~K. Kim and P.~T. Toan.
\newblock {Improved semidefinite programming bound on sizes of codes}.
\newblock {\em IEEE Trans. Inf. Theory}, 59(11):7337--7345, 2013.

\bibitem{K2016}
D.~Krioukov.
\newblock {Clustering Implies Geometry in Networks.}
\newblock {\em Phys. Rev. Lett.}, 116(20):208302, 2016.

\bibitem{sphere}
S.~Li.
\newblock {Concise formulas for the area and volume of a hyperspherical cap}.
\newblock {\em Asian J. Math. Stat.}, 4(1):66--70,
  2011.

\bibitem{LS89}
R.~Liptser and A.~Shiryaev.
\newblock {\em {Theory of Martingales}}.
\newblock Springer, 1989.

\bibitem{LFCDJCZ2001}
M.~D. Lukin, M.~Fleischhauer, R.~Cote, L.~M. Duan, D.~Jaksch, J.~I. Cirac, and
  P.~Zoller.
\newblock {Dipole Blockade and Quantum Information Processing in Mesoscopic
  Atomic Ensembles}.
\newblock {\em Phys. Rev. Lett.}, 87(3):037901, 2001.

\bibitem{N09}
M.~E.~J. Newman.
\newblock {Random Graphs with Clustering}.
\newblock {\em Phys. Rev. Lett.}, 103(5):058701, 2009.

\bibitem{N2010}
M.~E.~J. Newman.
\newblock Oxford University Press, 2010.

\bibitem{PTRW07}
G.~Pang, R.~Talreja, and W.~Whitt.
\newblock {Martingale proofs of many-server heavy-traffic limits for Markovian
  queues}.
\newblock {\em Probab. Surv.}, 4:193--267, 2007.

\bibitem{P2003}
M.~Penrose.
\newblock {\em {Random Geometric Graphs}}.
\newblock Oxford University Press, 2003.

\bibitem{PY02}
M.~D. Penrose and J.~Yukich.
\newblock {Limit theory for random sequential packing and deposition}.
\newblock {\em Ann. Appl. Probab.}, 12(1):272--301, 2002.

\bibitem{SWM2010}
M.~Saffman, T.~G. Walker, and K.~M{\o}lmer.
\newblock {Quantum information with Rydberg atoms}.
\newblock {\em Rev. Mod. Phys.}, 82(3):2313--2363, 2010.

\bibitem{SJK15}
J.~Sanders, M.~Jonckheere, and S.~Kokkelmans.
\newblock {Sub-Poissonian statistics of jamming limits in ultracold Rydberg
  gases.}
\newblock {\em Phys. Rev. Lett.}, 115(4):043002, 2015.

\bibitem{SBVK2014}
J.~Sanders, R.~van Bijnen, E.~Vredenbregt, and S.~Kokkelmans.
\newblock {Wireless network control of interacting Rydberg atoms}.
\newblock {\em Phys. Rev. Lett.}, 112(16):163001, apr 2014.

\bibitem{S2001}
C.~E. Shannon.
\newblock {A mathematical theory of communication}.
\newblock {\em SIGMOBILE Mob. Comput. Commun. Rev.}, 5(1):3--55, 2001.

\bibitem{S1967}
H.~Solomon.
\newblock {Random packing density}.
\newblock In {\em Proc. Fifth Berkeley Symp. on Prob. and Stat}, volume~3,
  pages 119--134, 1967.

\bibitem{SVV16SR}
C.~Stegehuis, R.~van~der Hofstad, and J.~S.~H. van Leeuwaarden.
\newblock {Epidemic spreading on complex networks with community structures}.
\newblock {\em Sci. Rep.}, 6:29748, 2016.

\bibitem{SVV16}
C.~Stegehuis, R.~van~der Hofstad, and J.~S.~H. van Leeuwaarden.
\newblock {Power-law relations in random networks with communities}.
\newblock {\em Phys. Rev. E}, 94(1):012302, 2016.

\bibitem{TST1991}
J.~Talbot, P.~Schaaf, and G.~Tarjus.
\newblock {Random sequential addition of hard spheres}.
\newblock {\em Mol. Phys.}, 72(6):1397--1406, 1991.

\bibitem{T2013}
S.~Torquato.
\newblock {\em {Random heterogeneous materials: microstructure and macroscopic
  properties}}, volume~16.
\newblock Springer-Verlag New York, 2013.

\bibitem{TS2010}
S.~Torquato and F.~H. Stillinger.
\newblock Jammed hard-particle packings: From kepler to bernal and beyond.
\newblock {\em Rev. Mod. Phys.}, 82:2633--2672, Sep 2010.

\bibitem{TUS2006}
S.~Torquato, O.~U. Uche, and F.~H. Stillinger.
\newblock {Random sequential addition of hard spheres in high Euclidean
  dimensions}.
\newblock {\em Phys. Rev. E}, 74(6):061308, 2006.

\bibitem{T07}
P.~Trapman.
\newblock {On analytical approaches to epidemics on networks}.
\newblock {\em Theor. Popul. Biol.}, 71(2):160--173, 2007.


\bibitem{W02}
W.~Whitt.
\newblock {\em {Stochastic-Process Limits}}.
\newblock Springer Series in Operations Research and Financial Engineering.
  Springer-Verlag, New York, 2002.

\bibitem{ZT2013}
G.~Zhang and S.~Torquato.
\newblock {Precise algorithm to generate random sequential addition of hard
  hyperspheres at saturation}.
\newblock {\em Phys. Rev. E}, 88(5):053312, 2013.

\end{thebibliography}

\end{document}